\documentclass[a4paper,12pt,leqno,dvipdfmx]{amsart}
\usepackage{latexsym}
\usepackage[all]{xy}

\usepackage{amssymb} 
\usepackage{amsmath} 
\usepackage{amscd}
\usepackage{color}
\usepackage{comment}
\usepackage{mathtools}
\usepackage{fancybox}

\usepackage{tikz}
\usetikzlibrary{patterns}
\usetikzlibrary{chains}
\usetikzlibrary{cd}
\definecolor{gray}{gray}{0.8}
\definecolor{Gray}{gray}{0.5}

\numberwithin{equation}{section} 

\newtheorem{theorem}{Theorem}[section]
\newtheorem{lemma}[theorem]{Lemma} 
\newtheorem{corollary}[theorem]{Corollary}
\newtheorem{proposition}[theorem]{Proposition} 
 
\newtheorem{remark}[theorem]{Remark}
\newtheorem{example}[theorem]{Example}
\newtheorem{definition}[theorem]{Definition}

\newtheorem*{theoremA*}{Theorem A}
\newtheorem*{theoremB*}{Theorem B}

\allowdisplaybreaks[4]

\def\C{\mathbb C}
\def\R{\mathbb R}
\def\Q{\mathbb Q}
\def\Z{\mathbb Z}
\def\P{\mathbb P}
\def\a{a}
\def\b{b}
\def\c{c}
\def\d{d}

\DeclareMathOperator{\wt}{wt}
\DeclareMathOperator{\Sym}{Sym}
\DeclareMathOperator{\Hom}{Hom}
\DeclareMathOperator{\Lie}{Lie}
\DeclareMathOperator{\pt}{pt}
\DeclareMathOperator{\sgn}{sgn}
\DeclareMathOperator{\Ind}{Ind}
\DeclareMathOperator{\height}{ht}
\DeclareMathOperator{\Spec}{Spec}
\DeclareMathOperator{\RSpec}{\bf{Spec}}
\DeclareMathOperator{\codim}{codim}
\DeclareMathOperator{\Pic}{Pic}
\DeclareMathOperator{\rank}{rank}
\DeclareMathOperator{\Hess}{Hess}
\DeclareMathOperator{\Vol}{Vol}
\newcommand{\into}{\hookrightarrow}
\newcommand{\onto}{\twoheadrightarrow}
\newcommand{\red}[1]{\textcolor{red}{#1}}

\newcommand{\w}[1]{w_{#1}}
\newcommand{\rb}[1]{k_{#1}}
\newcommand{\lb}[1]{k_{#1,-}}
\newcommand{\rp}[1]{p_{#1}}
\newcommand{\lp}[1]{p_{#1,-}}
\newcommand{\m}[1]{m_{#1}}
\newcommand{\q}[1]{q_{#1}}
\newcommand{\LL}[1]{L_{#1}}
\newcommand{\DD}[1]{D({#1})}
\newcommand{\Sn}{\mathfrak{S}_n}

\newcommand{\Pet}[1]{Pet_{#1}}
\newcommand{\Perm}[1]{Perm_{#1}}

\begin{document}

\title[Peterson Schubert calculus in type A]{Geometry of Peterson Schubert calculus in type A and left-right diagrams}
\author {Hiraku Abe}
\address{Faculty of Science, Department of Applied Mathematics, Okayama University of
Science, 1-1 Ridai-cho, Kita-ku, Okayama, 700-0005, Japan}
\email{hirakuabe@globe.ocn.ne.jp}

\author {Tatsuya Horiguchi}
\address{Osaka City University Advanced Mathematical Institute, Sumiyoshi-ku, Osaka 558-8585, Japan}
\email{tatsuya.horiguchi0103@gmail.com}

\author {Hideya Kuwata}
\address{Kindai University Technical College, 7-1 Kasugaoka, Nabari, Mie 518-0459, Japan}
\email{hideya0813@gmail.com}

\author {Haozhi Zeng}
\address{School of Mathematics and Statistics, Huazhong University of Science and Technology, Wuhan, 430074, P.R. China}
\email{zenghaozhi@icloud.com} 

\begin{abstract}
We introduce an additive basis of the integral cohomology ring of the Peterson variety which reflects the geometry of certain subvarieties of the Peterson variety. We explain the positivity of the structure constants from a geometric viewpoint, and provide a manifestly positive combinatorial formula for them. We also prove that our basis coincides with the additive basis introduced by Harada-Tymoczko.
\end{abstract}

\maketitle

\section{Introduction}
Let $n$ be a positive integer and $Fl_n$ the full-flag variety of $\C^n$. Namely, $Fl_n$ is the collection of nested sequences of linear subspaces of $\C^n$ given as follows:
\begin{align*}
 Fl_n = \{V_{\bullet}=(V_1\subset V_2\subset \cdots \subset V_{n}=\C^n) \mid \dim_{\C}V_i=i \ (1\le i\le n)\}.
\end{align*}
Let $N$ be an $n\times n$ regular nilpotent matrix viewed as a linear map $N\colon \C^n\rightarrow \C^n$. The Peterson variety $\Pet{n}\subseteq Fl_n$ is defined by
\begin{align*}
 \Pet{n} \coloneqq \{V_{\bullet} \in Fl_n \mid NV_i\subseteq V_{i+1} \ (1\le i\le n-1)\},
\end{align*}
where $NV_i$ denotes the image of $V_i$ under the map $N\colon \C^n\rightarrow \C^n$. It was introduced by Dale Peterson to study the quantum cohomology ring of $Fl_n$, and it has been appeared in several contexts (e.g.\ \cite{AHMMS,Ba,ha-ho-ma,Kostant,Rietsch06}).

For a permutation $w\in\mathfrak{S}_n$, let $X_w\subseteq Fl_n$ be the Schubert variety associated with $w$, and $\Omega_w\subseteq Fl_n$ the dual Schubert variety associated with $w$. We denote by $[n-1]$ the set of integers $1,2,\ldots,n-1$. For a subset $J\subseteq[n-1]$, let $w_J\in\mathfrak{S}_n$ be the longest element of the Young subgroup $\mathfrak{S}_J$ of the permutation group $\Sn$ associated with $J$ (see Section~\ref{subsect: combi} for details), and set
\begin{align*}
X_{J} 
\coloneqq X_{w_J}\cap \Pet{n} \quad \text{and} \quad \Omega_{J} 
\coloneqq \Omega_{w_J}\cap \Pet{n}.
\end{align*}
Then $X_J$ and $\Omega_J$ in $\Pet{n}$ play similar roles to that of Schubert varieties and dual Schubert varieties in $Fl_n$, and provide important information on the topology of $\Pet{n}$.

In this paper, we construct an additive basis $\{\varpi_J\mid J\subseteq[n-1]\}$ of the integral cohomology ring $H^*(\Pet{n};\Z)$ which reflects the geometry of $X_J$ and $\Omega_J$ (Theorem~\ref{thm: Z-basis of cohomology of Peterson}). As a consequence, we may consider the structure constants for the multiplication rule:
\begin{align}\label{intro:100}
 \varpi_J\cdot \varpi_K = \sum_{L\subseteq[n-1]} d_{JK}^L \varpi_L , \qquad d_{JK}^L\in \Z.
\end{align}
It turns out that all $d_{JK}^L$ are non-negative integers, and we give a geometric proof of this positivity (Proposition~\ref{prop: positivity}). We also provide a manifestly positive combinatorial formula for $d_{JK}^L$ (Theorem~\ref{theorem:LeftRight_diagram}) in terms of \textit{left-right diagrams}, which we introduce in this paper. 

To find our formula for $d_{JK}^L$, 
we prove several properties of the cohomology classes $\varpi_J$ which are inherited from the geometry of $X_J$ and $\Omega_J$.
In particular, writing $\Omega_J$ as an intersection of divisors on $\Pet{n}$ provides the geometric idea behind our formula for $d_{JK}^L$ in terms of left-right diagrams.

We also show that our basis $\{\varpi_J\mid J\subseteq[n-1]\}$ coincides with the additive basis of the cohomology ring $H^*(\Pet{n};\C)$ with $\C$-coefficients introduced by Harada-Tymoczko (\cite{HaTy}). Their basis is obtained by taking restriction of certain Schubert classes to $\Pet{n}$, and it is called the \textit{Peterson Schubert basis}. It has been studied by Bayegan-Harada (\cite{BaHa}), Drellich (\cite{Dre2}) and Goldin-Gorbutt (\cite{GoGo}). In \cite{GoGo}, Goldin and Gorbutt gave combinatorial formulas for the structure constants of Harada-Tymoczko's basis (in a certain equivariant setting) which are manifestly positive and integral. 
Thus, after taking the non-equivariant limit, their formulas and ours both describe the same structure constants, but these formulas have different perspectives; 
their approach is mostly combinatorial whereas our approach is based on the geometry of $X_J$ and $\Omega_J$. We include a short comparison of their formulas and ours in Section~\ref{sec: relations to other works}.

Interestingly, our computations match with those of Berget-Spink-Tseng (\cite[Sect.~7]{Berget-Spink-Tseng}) on a certain subring of the cohomology ring of a toric variety which is called the permutohedral variety.
One of their results can be interpreted as a formula describing the structure constants $d_{JK}^L$ as products of \textit{mixed Eulerian numbers} which were introduced and studied by Postnikov (\cite{Postnikov09}).
With this connection in mind, our formula (Theorem~\ref{theorem:LeftRight_diagram}) for $d_{JK}^L$ can also be thought as computing some products of mixed Eulerian numbers by using the geometry of $\Pet{n}$.
We explain this connection in Section~\ref{sec: relations to other works} including the relations with the works of Nadeau-Tewari (\cite{Nadeau-Tewari}) and the second author (\cite{Horiguchi21}).

\bigskip
\noindent \textbf{Acknowledgments}.
We are grateful to Mikiya Masuda for his support and encouragement. This research is supported in part by Osaka City University Advanced Mathematical Institute (MEXT Joint Usage/Research Center on Mathematics and Theoretical Physics): the topology and combinatorics of Hessenberg varieties. The first author is supported in part by JSPS Grant-in-Aid for Early-Career Scientists: 18K13413. The second author is supported in part by JSPS Grant-in-Aid for Young Scientists: 19K14508. The fourth author is supported in part by NSFC: 11901218.

\bigskip

%%%%%%%%%%%%%%%%%%%%%%%%%%%%%
%%%%%%%%%%%%%%%%%%%%%%%%%%%%%
\section{Basic notations}\label{sec: Basic notations}
%%%%%%%%%%%%%%%%%%%%%%%%%%%%%
%%%%%%%%%%%%%%%%%%%%%%%%%%%%%
In this section, we recall some terminologies which will be used in this paper.

\subsection{Combinatorics on the Dynkin diagram of type A}\label{subsect: combi}
Let $n(\ge2)$ be a positive integer. We use the notation $[n-1]\coloneqq\{1,2,\ldots,n-1\}$, and we regard it as the set of vertices of the Dynkin diagram of type $A_{n-1}$ for the rest of the paper. Namely, two vertices $i,j\in[n-1]$ are connected by an edge if and only if $|i-j|=1$. See Figure~\ref{pic: Dynkin diagram}.

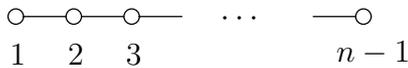
\begin{figure}[htbp]
\[
%WinTpicVersion4.32a
{\unitlength 0.1in%
\begin{picture}(18.8000,2.0000)(15.6000,-15.2000)%
% CIRCLE 2 0 3 0 Black White  
% 4 1600 1400 1640 1400 1640 1400 1640 1400
% 
\special{pn 8}%
\special{ar 1600 1400 40 40 0.0000000 6.2831853}%
% CIRCLE 2 0 3 0 Black White  
% 4 1900 1400 1940 1400 1940 1400 1940 1400
% 
\special{pn 8}%
\special{ar 1900 1400 40 40 0.0000000 6.2831853}%
% CIRCLE 2 0 3 0 Black White  
% 4 2200 1400 2240 1400 2240 1400 2240 1400
% 
\special{pn 8}%
\special{ar 2200 1400 40 40 0.0000000 6.2831853}%
% LINE 2 0 3 0 Black White  
% 2 2160 1400 1940 1400
% 
\special{pn 8}%
\special{pa 2160 1400}%
\special{pa 1940 1400}%
\special{fp}%
% LINE 2 0 3 0 Black White  
% 2 1860 1400 1640 1400
% 
\special{pn 8}%
\special{pa 1860 1400}%
\special{pa 1640 1400}%
\special{fp}%
% CIRCLE 2 0 3 0 Black White  
% 4 3400 1400 3440 1400 3440 1400 3440 1400
% 
\special{pn 8}%
\special{ar 3400 1400 40 40 0.0000000 6.2831853}%
% LINE 2 0 3 0 Black White  
% 2 3360 1400 3140 1400
% 
\special{pn 8}%
\special{pa 3360 1400}%
\special{pa 3140 1400}%
\special{fp}%
% LINE 2 0 3 0 Black White  
% 2 2460 1400 2240 1400
% 
\special{pn 8}%
\special{pa 2460 1400}%
\special{pa 2240 1400}%
\special{fp}%
% STR 2 0 3 0 Black White  
% 4 2660 1350 2660 1450 2 0 0 0
% $\cdots$
\put(26.6000,-14.5000){\makebox(0,0)[lb]{$\cdots$}}%
% STR 2 0 3 0 Black White  
% 4 1570 1550 1570 1650 2 0 0 0
% $1$
\put(15.7000,-16.5000){\makebox(0,0)[lb]{$1$}}%
% STR 2 0 3 0 Black White  
% 4 2170 1550 2170 1650 2 0 0 0
% $3$
\put(21.7000,-16.5000){\makebox(0,0)[lb]{$3$}}%
% STR 2 0 3 0 Black White  
% 4 1870 1550 1870 1650 2 0 0 0
% $2$
\put(18.7000,-16.5000){\makebox(0,0)[lb]{$2$}}%
% STR 2 0 3 0 Black White  
% 4 3260 1550 3260 1650 2 0 0 0
% $n-1$
\put(32.6000,-16.5000){\makebox(0,0)[lb]{$n-1$}}%
\end{picture}}%
\]
\caption{The Dynkin diagram of type $A_{n-1}$.}
\label{pic: Dynkin diagram}
\end{figure}

We regard each subset $J\subseteq[n-1]$ as a full-subgraph of the Dynkin diagram appeared above. We may decompose it into the connected components: 
\begin{align}\label{eq: connected components}
 J = J_1 \sqcup J_2 \sqcup \cdots \sqcup J_{m},
\end{align}
where each $J_k\ (1\le k\le m)$ is the set of vertices of a maximal connected subgraph of $J$. To determine each $J_k$ uniquely, we require that elements of $J_k$ are less than elements of $J_{k'}$ when $k< k'$.

\begin{example}
Let $n=10$ and $J=\{1,2,4,5,6,9\}$. Then we have 
\begin{align*}
 J = \{1,2\} \sqcup\{4,5,6\} \sqcup\{9\} = J_1\sqcup J_2\sqcup J_3.
\end{align*}
\end{example}

\vspace{5pt}

For $J\subseteq[n-1]$, let us introduce a Young subgroup given by
\begin{align*}
 \mathfrak{S}_J \coloneqq \mathfrak{S}_{J_1}\times \mathfrak{S}_{J_2}\times \cdots \mathfrak{S}_{J_m}
 \subseteq \Sn,
\end{align*}
where each $\mathfrak{S}_{J_k}$ is the subgroup of $\Sn$ generated by the simple reflections $s_i$ for all $i\in J_k$. Let $w_J$ be the longest element of $\mathfrak{S}_J$, i.e., \ 
\begin{align}\label{eq: def of wJ}
 w_J \coloneqq w_0^{(J_1)} w_0^{(J_2)} \cdots w_0^{(J_{m})} \in \mathfrak{S}_J, 
\end{align}
where each $w_0^{(J_k)}$ is the longest element of the permutation group $\mathfrak{S}_{J_k}$ $(1\le k\le m)$. 

\begin{example}\label{eg: wJ}
{\rm
Let $n=10$ and $J=\{1,2\} \sqcup\{4,5,6\} \sqcup\{9\}$ as above. By identifying the permutation $w_J$ with its permutation matrix, we have 
\begin{align*}
 w_J 
 = w_0^{(J_1)} w_0^{(J_2)} w_0^{(J_3)} 
 = (s_1s_2s_1) (s_4s_5s_6s_4s_5s_4) (s_9)
 =
 \left(
 \begin{array}{@{\,}ccc|cccc|c|cc@{\,}}
     & & 1 & & & & & & & \\
     & 1 & & & & & & & & \\ 
    1 & & & & & & & & & \\ \hline
     & & & & & & 1& & &  \\
     & & & & & 1 & & & & \\
     & & & & 1 & & & & & \\
     & & & 1 & & & & & & \\ \hline
     & & & & & & & 1 & & \\ \hline
     & & & & & & & & & 1 \\
     & & & & & & & & 1 & 
 \end{array}
 \right).
\end{align*}}
\end{example}

\vspace{10pt}

We can identify each permutation $w\in\Sn$ with its permutation flag $V_{\bullet}\in Fl_n$ defined by $V_i =\langle e_{w(1)}, e_{w(2)},\ldots, e_{w(i)}\rangle$ for $1\le i\le n$, where $e_1,e_2,\ldots,e_n$ denotes the standard basis of $\C^n$, and the right hand side is the linear subspace of $\C^n$ spanned by $e_{w(1)}, e_{w(2)},\ldots, e_{w(i)}$. Using this identification, we explain how the permutations $w_J$ are related to the Peterson variety. Let ${\rm GL}_n(\C)$ be the general linear group of invertible $n\times n$ complex matrices. Let $T\subseteq {\rm GL}_n(\C)$ be the maximal torus consisting of diagonal matrices. Let us identify $\C^{\times}$ with a subgroup of $T$ as follows: 
\begin{align}\label{eq: C-starsubgroup}
\C^{\times} = 
\left\{
\left.
\begin{pmatrix}
g & & & \\
 & g^2 & & \\
 & & \ddots & \\
 & & & g^n 
\end{pmatrix}
\in T\ 
\right| \ g\in \C^{\times}
\right\}.
\end{align}
The flag variety $Fl_n$ admits a natural ${\rm GL}_n(\C)$-action by regarding each element $g\in {\rm GL}_n(\C)$ as an automorphism $g\colon \C^n\rightarrow \C^n$. Restricting this ${\rm GL}_n(\C)$-action on $Fl_n$ to the above subgroup $\C^{\times}$, it is well-known that the fixed point set $(Fl_n)^{\C^{\times}}$ is the set of the permutation flags, i.e.\ $(Fl_n)^{\C^{\times}}=\Sn$ (e.g.\ \cite[The proof of Lemma~2 in Sect.~10.1]{fult97}). It is straightforward to see that this $\C^{\times}$-action on $Fl_n$ preserves $\Pet{n}$, and it was shown in \cite{HaTy} that the fixed point set $(\Pet{n})^{\C^{\times}}$ is given by
\begin{align}\label{eq: fixed points of Pet}
(\Pet{n})^{\C^{\times}} =\{w_J \in\Sn \mid J\subseteq [n-1]\}.
\end{align}
Because of this relation, the combinatorics of $w_J$ will be important to understand the structure of $\Pet{n}$.

For $1\le i\le n-1$, let $s_i\in\mathfrak{S}_n$ be the simple reflection which interchanges $i$ and $i+1$. We denote by $\le$ the Bruhat order on $\mathfrak{S}_n$, that is, we have $u\le v$ $(u,v\in\mathfrak{S}_n)$ if and only if a reduced expression of $u$ is a subword of a reduced expression of $v$.
\begin{lemma}\label{lem: So-san's lemma 1}
For $J\subseteq[n-1]$ and $1\le i\le n-1$, we have 
$s_i\le w_J$ if and only if $i\in J$.
\end{lemma}

\begin{proof}
Recall that $w_J$ is the product of longest elements in $\mathfrak{S}_{J_k}$ for $1\le k\le m$:
\begin{align*}
 w_J = w_0^{(J_1)} w_0^{(J_2)} \cdots w_0^{(J_{m})}.
\end{align*}
Since each $w_0^{(J_k)}$ ($1\le k\le m$) preserves the decomposition \eqref{eq: connected components}, it follows that the length of $w_J$ is the same as the sum of the length of $w_0^{(J_k)}$ for $1\le k\le m$. Thus, the products of reduced expressions of $w_0^{(J_k)}$ for $1\le k\le m$ give a reduced expression of $w_J$. Here, an arbitrary reduced expression of $w_0^{(J_k)}$ contains a simple reflection $s_i$ if and only if $i\in J_k$. Therefore, it follows that a reduced expression of $w_J$ contains $s_i$ if and only if $i\in J$ (see Example~\ref{eg: wJ}). This implies the desired claim. 
\end{proof}

The following claim appears in \cite[Lemma 6]{Insko-Tymoczko}, but we give a proof for the reader.
\begin{lemma}\label{lem: So-san's lemma}
For $J,J'\subseteq[n-1]$, we have
\begin{align}\label{eq: subavr 15}
w_{J'} \le w_{J} \quad \text{if and only if} \quad J'\subseteq J .
\end{align}
\end{lemma}

\begin{proof}
We first prove that $J'\subseteq J$ under the assumption $w_{J'} \le w_{J}$. For this, take an arbitrary element $i\in J'$. By the previous lemma, we have $s_i\le w_{J'}$. Combining this with the assumption $w_{J'} \le w_{J}$, we obtain that $s_i\le w_{J}$. Thus, it follows that $i\in J$ by the previous lemma again. 

We next prove that $w_{J'} \le w_{J}$ under the assumption $J'\subseteq J$. Take the decomposition 
\begin{align*}
 J' = J'_1 \sqcup J'_2 \sqcup \cdots \sqcup J'_{m'}
\end{align*}
into the connected components as in \eqref{eq: connected components}. Since each $J'_{\ell}\ (1\le \ell \le m')$ is connected, it is contained in some connected component $J_k$ of $J$. This leads us to consider a map
\begin{align*}
 \varphi \colon \{1,2,\ldots,m'\} \rightarrow \{1,2,\ldots,m\}
\end{align*}
which we define by the conditions $J'_i\subseteq J_{\varphi(i)}$ for $1\le i\le m'$. Then we have that
\begin{align*}
 \bigsqcup_{\varphi(i)=k} J'_i \subseteq J_{k}
\end{align*}
by the definition of the map $\varphi$. This implies that 
\begin{align*}
 \prod_{\varphi(i)=k} w_0^{(J'_i)} \le w_0^{(J_k)}
 \qquad \text{in $\mathfrak{S}_{J_k}$}
\end{align*}
since $w_0^{(J_k)}$ is the longest permutation in $\mathfrak{S}_{J_k}$. Recalling that each $J_k$ is a connected component of $J$, these inequalities for $1\le k\le m$ imply that $w_{J'}\le w_{J}$.
\end{proof}

\begin{example}
\emph{
Let $n=8$, $J'=\{1,4,5,7\}$ and $J=\{1,2,4,5,6,7\}$ so that we have $J'\subseteq J$. In this case, we have
\begin{align*}
 J'=\{1\}\sqcup\{4,5\}\sqcup\{7\}=J'_1\sqcup J'_2\sqcup J'_3
 \quad \text{and} \quad
 J=\{1,2\}\sqcup\{4,5,6,7\}=J_1\sqcup J_2,
\end{align*}
and hence
\begin{align*}
 w_{J'}= (s_1)(s_4s_5s_4)(s_7) \le (s_1s_2s_1)(s_4s_5s_6s_7s_4s_5s_6s_4s_5s_4)=w_J.
\end{align*}
}
\end{example}

\subsection{Hessenberg varieties}\label{subsect: Hessenberg}
In this subsection, we briefly recall the notion of Hessenberg varieties. They are (possibly reducible) subvarieties of the flag variety $Fl_n$, and will appear in the next section. A function $h\colon [n] \to [n]$ is a \textit{Hessenberg function} if it satisfies the following two conditions:
\begin{enumerate}
\item[(i)] $h(1) \leq h(2) \leq \cdots \leq h(n)$,
\item[(ii)] $h(j) \geq j$ for all $j \in [n]$. 
\end{enumerate}
Note that $h(n)=n$ by definition. We may identify a Hessenberg function $h$ with a configuration of (shaded) boxes on a square grid of size $n \times n$ which consists of boxes in the $i$-th row and the $j$-th column satisfying $i \leq h(j)$ for $i, j\in[n]$, as we illustrate in the following example.

\begin{example}\label{example:HessenbergFunction}
\emph{
Let $n=5$. The Hessenberg function $h\colon[5]\rightarrow [5]$ given by 
\[(h(1),h(2),h(3),h(4),h(5))=(2,3,3,5,5)\]
corresponds to the configuration of the shaded boxes drawn in Figure $\ref{pic:stair-shape}$.}
\begin{figure}[h]
\begin{center}
\begin{picture}(75,75)
\put(0,63){\colorbox{gray}}
\put(0,67){\colorbox{gray}}
\put(0,72){\colorbox{gray}}
\put(4,63){\colorbox{gray}}
\put(4,67){\colorbox{gray}}
\put(4,72){\colorbox{gray}}
\put(9,63){\colorbox{gray}}
\put(9,67){\colorbox{gray}}
\put(9,72){\colorbox{gray}}

\put(15,63){\colorbox{gray}}
\put(15,67){\colorbox{gray}}
\put(15,72){\colorbox{gray}}
\put(19,63){\colorbox{gray}}
\put(19,67){\colorbox{gray}}
\put(19,72){\colorbox{gray}}
\put(24,63){\colorbox{gray}}
\put(24,67){\colorbox{gray}}
\put(24,72){\colorbox{gray}}

\put(30,63){\colorbox{gray}}
\put(30,67){\colorbox{gray}}
\put(30,72){\colorbox{gray}}
\put(34,63){\colorbox{gray}}
\put(34,67){\colorbox{gray}}
\put(34,72){\colorbox{gray}}
\put(39,63){\colorbox{gray}}
\put(39,67){\colorbox{gray}}
\put(39,72){\colorbox{gray}}

\put(45,63){\colorbox{gray}}
\put(45,67){\colorbox{gray}}
\put(45,72){\colorbox{gray}}
\put(49,63){\colorbox{gray}}
\put(49,67){\colorbox{gray}}
\put(49,72){\colorbox{gray}}
\put(54,63){\colorbox{gray}}
\put(54,67){\colorbox{gray}}
\put(54,72){\colorbox{gray}}

\put(60,63){\colorbox{gray}}
\put(60,67){\colorbox{gray}}
\put(60,72){\colorbox{gray}}
\put(64,63){\colorbox{gray}}
\put(64,67){\colorbox{gray}}
\put(64,72){\colorbox{gray}}
\put(69,63){\colorbox{gray}}
\put(69,67){\colorbox{gray}}
\put(69,72){\colorbox{gray}}

\put(0,48){\colorbox{gray}}
\put(0,52){\colorbox{gray}}
\put(0,57){\colorbox{gray}}
\put(4,48){\colorbox{gray}}
\put(4,52){\colorbox{gray}}
\put(4,57){\colorbox{gray}}
\put(9,48){\colorbox{gray}}
\put(9,52){\colorbox{gray}}
\put(9,57){\colorbox{gray}}

\put(15,48){\colorbox{gray}}
\put(15,52){\colorbox{gray}}
\put(15,57){\colorbox{gray}}
\put(19,48){\colorbox{gray}}
\put(19,52){\colorbox{gray}}
\put(19,57){\colorbox{gray}}
\put(24,48){\colorbox{gray}}
\put(24,52){\colorbox{gray}}
\put(24,57){\colorbox{gray}}

\put(30,48){\colorbox{gray}}
\put(30,52){\colorbox{gray}}
\put(30,57){\colorbox{gray}}
\put(34,48){\colorbox{gray}}
\put(34,52){\colorbox{gray}}
\put(34,57){\colorbox{gray}}
\put(39,48){\colorbox{gray}}
\put(39,52){\colorbox{gray}}
\put(39,57){\colorbox{gray}}

\put(45,48){\colorbox{gray}}
\put(45,52){\colorbox{gray}}
\put(45,57){\colorbox{gray}}
\put(49,48){\colorbox{gray}}
\put(49,52){\colorbox{gray}}
\put(49,57){\colorbox{gray}}
\put(54,48){\colorbox{gray}}
\put(54,52){\colorbox{gray}}
\put(54,57){\colorbox{gray}}

\put(60,48){\colorbox{gray}}
\put(60,52){\colorbox{gray}}
\put(60,57){\colorbox{gray}}
\put(64,48){\colorbox{gray}}
\put(64,52){\colorbox{gray}}
\put(64,57){\colorbox{gray}}
\put(69,48){\colorbox{gray}}
\put(69,52){\colorbox{gray}}
\put(69,57){\colorbox{gray}}

\put(15,33){\colorbox{gray}}
\put(15,37){\colorbox{gray}}
\put(15,42){\colorbox{gray}}
\put(19,33){\colorbox{gray}}
\put(19,37){\colorbox{gray}}
\put(19,42){\colorbox{gray}}
\put(24,33){\colorbox{gray}}
\put(24,37){\colorbox{gray}}
\put(24,42){\colorbox{gray}}

\put(30,33){\colorbox{gray}}
\put(30,37){\colorbox{gray}}
\put(30,42){\colorbox{gray}}
\put(34,33){\colorbox{gray}}
\put(34,37){\colorbox{gray}}
\put(34,42){\colorbox{gray}}
\put(39,33){\colorbox{gray}}
\put(39,37){\colorbox{gray}}
\put(39,42){\colorbox{gray}}

\put(45,33){\colorbox{gray}}
\put(45,37){\colorbox{gray}}
\put(45,42){\colorbox{gray}}
\put(49,33){\colorbox{gray}}
\put(49,37){\colorbox{gray}}
\put(49,42){\colorbox{gray}}
\put(54,33){\colorbox{gray}}
\put(54,37){\colorbox{gray}}
\put(54,42){\colorbox{gray}}

\put(60,33){\colorbox{gray}}
\put(60,37){\colorbox{gray}}
\put(60,42){\colorbox{gray}}
\put(64,33){\colorbox{gray}}
\put(64,37){\colorbox{gray}}
\put(64,42){\colorbox{gray}}
\put(69,33){\colorbox{gray}}
\put(69,37){\colorbox{gray}}
\put(69,42){\colorbox{gray}}

\put(45,18){\colorbox{gray}}
\put(45,22){\colorbox{gray}}
\put(45,27){\colorbox{gray}}
\put(49,18){\colorbox{gray}}
\put(49,22){\colorbox{gray}}
\put(49,27){\colorbox{gray}}
\put(54,18){\colorbox{gray}}
\put(54,22){\colorbox{gray}}
\put(54,27){\colorbox{gray}}

\put(60,18){\colorbox{gray}}
\put(60,22){\colorbox{gray}}
\put(60,27){\colorbox{gray}}
\put(64,18){\colorbox{gray}}
\put(64,22){\colorbox{gray}}
\put(64,27){\colorbox{gray}}
\put(69,18){\colorbox{gray}}
\put(69,22){\colorbox{gray}}
\put(69,27){\colorbox{gray}}

\put(45,3){\colorbox{gray}}
\put(45,7){\colorbox{gray}}
\put(45,12){\colorbox{gray}}
\put(49,3){\colorbox{gray}}
\put(49,7){\colorbox{gray}}
\put(49,12){\colorbox{gray}}
\put(54,3){\colorbox{gray}}
\put(54,7){\colorbox{gray}}
\put(54,12){\colorbox{gray}}

\put(60,3){\colorbox{gray}}
\put(60,7){\colorbox{gray}}
\put(60,12){\colorbox{gray}}
\put(64,3){\colorbox{gray}}
\put(64,7){\colorbox{gray}}
\put(64,12){\colorbox{gray}}
\put(69,3){\colorbox{gray}}
\put(69,7){\colorbox{gray}}
\put(69,12){\colorbox{gray}}

\put(0,0){\framebox(15,15)}
\put(15,0){\framebox(15,15)}
\put(30,0){\framebox(15,15)}
\put(45,0){\framebox(15,15)}
\put(60,0){\framebox(15,15)}
\put(0,15){\framebox(15,15)}
\put(15,15){\framebox(15,15)}
\put(30,15){\framebox(15,15)}
\put(45,15){\framebox(15,15)}
\put(60,15){\framebox(15,15)}
\put(0,30){\framebox(15,15)}
\put(15,30){\framebox(15,15)}
\put(30,30){\framebox(15,15)}
\put(45,30){\framebox(15,15)}
\put(60,30){\framebox(15,15)}
\put(0,45){\framebox(15,15)}
\put(15,45){\framebox(15,15)}
\put(30,45){\framebox(15,15)}
\put(45,45){\framebox(15,15)}
\put(60,45){\framebox(15,15)}
\put(0,60){\framebox(15,15)}
\put(15,60){\framebox(15,15)}
\put(30,60){\framebox(15,15)}
\put(45,60){\framebox(15,15)}
\put(60,60){\framebox(15,15)}
\end{picture}
\end{center}
\caption{The configuration of the shaded boxes corresponding to $h$.}
\label{pic:stair-shape}
\end{figure}
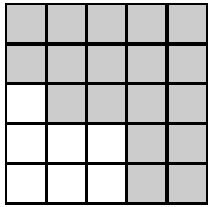
\end{example}

For an $n\times n$ matrix $X$ considered as a linear map $X\colon\C^n \to \C^n$ and a Hessenberg function $h\colon [n] \to [n]$, the \textit{Hessenberg variety} associated with $X$ and $h$ (\cite{De Mari-Procesi-Shayman,ty}) is defined as
\begin{equation} \label{eq:DefinitionHessenbergVariety}
\Hess(X,h)=\{V_\bullet \in Fl_n \mid XV_j \subseteq V_{h(j)} \ \mbox{for all} \ j\in[n] \}.
\end{equation}
Let $N$ be the $n\times n$ regular nilpotent matrix in Jordan canonical form, i.e., 
\begin{align}\label{eq: def of N}
N =
\begin{pmatrix} 
0 & 1 & & &\\ 
& 0 & 1 & &\\
&   & \ddots & \ddots &\\
&  & & 0 & 1\\
&  & & & 0
\end{pmatrix}.
\end{align}
Then $\Hess(N,h)$ is called a \textit{regular nilpotent Hessenberg variety}. It is known that $\Hess(N,h)$ is an irreducible projective variety of (complex) dimension $\sum_{i=1}^n (h(i)-i)$ (\cite[Lemma 7.1]{an-ty} and \cite{ty}). When $h(i)=i+1$ for all $1\le i\le n-1$, it is clear that $\Hess(N,h)=\Pet{n}$ by definition. 
In particular, we obtain the well-known formula 
\[\dim_{\C}\Pet{n}=n-1.\]
For Hessenberg functions $h,h'\colon[n]\rightarrow [n]$, it is clear that if $h(i)\le h'(i)$ for $1\le i\le n$ then $\Hess(N,h)\subseteq \Hess(N,h')$. For example, the Hessenberg function $h\colon[5]\rightarrow[5]$ given in Example~\ref{example:HessenbergFunction} defines a $3$-dimensional regular nilpotent Hessenberg variety $\Hess(N,h)$ which is contained in $\Pet{5}(\subseteq Fl_5)$.

\bigskip

%%%%%%%%%%%%%%%%%%%%%%%%%%%%%
%%%%%%%%%%%%%%%%%%%%%%%%%%%%%
\section{Geometric constructions}\label{sec: geometric constructions}
%%%%%%%%%%%%%%%%%%%%%%%%%%%%%
%%%%%%%%%%%%%%%%%%%%%%%%%%%%%
In this section, we introduce two kinds of subvarieties $X_J$ and $\Omega_J$ in $\Pet{n}$ for each $J\subseteq[n-1]$, and we establish geometric properties of them. They will play important roles to construct an additive basis of the integral cohomology ring $H^*(\Pet{n};\Z)$ in the next section.

\subsection{Analogue of Schubert varieties in the Peterson variety}
For $w\in \mathfrak{S}_n$, 
let $X_{w}\subseteq Fl_n$ be the Schubert variety associated with $w$ and $\Omega_w\subseteq Fl_n$ the dual Schubert variety associated with $w$ (\cite[Sect.~10]{fult97}).
We note that $\dim_{\C} X_{w} =\codim_{\C} (\Omega_w,Fl_n)= \ell(w)$, where $\ell(w)$ is the length of $w$.
\begin{definition}
For $J\subseteq [n-1]$, we define
\begin{equation}
\begin{split}\label{eq: subavr 10}
X_{J} 
\coloneqq X_{w_J}\cap \Pet{n} \quad \text{and} \quad \Omega_{J} 
\coloneqq \Omega_{w_J}\cap \Pet{n},
\end{split}
\end{equation}
where $w_J\in\mathfrak{S}_n$ is the permutation defined in $\eqref{eq: def of wJ}$.
\end{definition}

Peterson (\cite{Peterson}) studied a particular open affine subset of $\Omega_J$ to construct the quantum cohomology ring of $Fl_n$ (cf.\ \cite{Kostant,Rietsch06}). Also, Insko (\cite{Insko}) and Insko-Tymoczko (\cite{Insko-Tymoczko}) studied $X_J$ to show the injectivity of the homomorphism $H_*(\Pet{n};\Z)\rightarrow H_*(Fl_n;\Z)$. It turns out that $X_J$ and $\Omega_J$ in $\Pet{n}$ play similar roles to that of Schubert varieties and dual Schubert varieties in $Fl_n$. As an illustrating property, we begin with the following claim. Recall that we have $X_{w}\cap \Omega_{v}\ne\emptyset$ in $Fl_n$ if and only if $w\ge v$. 

\begin{proposition}\label{prop: point intersection 2}
For $J,J'\subseteq[n-1]$, we have
\begin{align*}
X_{J}\cap \Omega_{J'} \neq \emptyset
\quad \text{if and only if} \quad
J \supseteq J'.
\end{align*}
Moreover, when $J=J'$, we have $X_J \cap \Omega_J = \{w_{J}\}$.
\end{proposition}

\begin{proof}
If $X_{J}\cap \Omega_{J'} \neq \emptyset$, then we have $(X_{w_J}\cap\Omega_{w_{J'}})\cap \Pet{n}\ne\emptyset$ by definition. Note that $(X_{w_J}\cap\Omega_{w_{J'}})\cap \Pet{n}$ is complete, and it is preserved by the $\C^{\times}$-action on $\Pet{n}$ described in Section~\ref{sec: Basic notations}. Thus, it follows that it contains a $\C^{\times}$-fixed point (e.g.\ \cite[Chap.\ VIII, Sect.\ 21.2]{Humphreys}). Since we have
\begin{align*}
((X_{w_J}\cap\Omega_{w_{J'}})\cap \Pet{n})^{\C^{\times}} 
= (X_{w_J}\cap\Omega_{w_{J'}})^{\C^{\times}} \cap (\Pet{n})^{\C^{\times}},
\end{align*}
we see that there exists a $\C^{\times}$-fixed point $w_K\in \Pet{n}$ (see \eqref{eq: fixed points of Pet}) such that $w_K\in X_{w_J}$ and $w_K\in\Omega_{w_{J'}}$. The former condition implies that $w_J\ge w_K$, and the latter condition implies that $w_K \ge w_{J'}$ (\cite[Sect.\ 10.2 and 10.5]{fult97}). Thus, we obtain $w_J\ge w_{J'}$, and it follows that $J\supseteq J'$ from Lemma~\ref{lem: So-san's lemma}.

If $J \supseteq J'$, then we have $w_J\ge w_{J'}$ by Lemma~\ref{lem: So-san's lemma}. This implies that $w_{J}\in X_{w_J}\cap\Omega_{w_{J'}}$, and hence $X_{J}\cap \Omega_{J'}\ne\emptyset$ follows.
\end{proof}

A distinguished property of $X_J$ is that it is a regular nilpotent Hessenberg variety for a certain Hessenberg function. Let us explain this in the following. For each $J\subseteq[n-1]$, there is a natural Hessenberg function which is determined by $J$ as follows. Let $h_J\colon[n]\rightarrow[n]$ be a function given by
\begin{align}\label{eq: def of hJ}
 h_J (i) =
 \begin{cases}
  i+1 \quad &\text{if $i\in J$}\\
  i &\text{if $i\notin J$}
 \end{cases}
\end{align}
for $1\le i\le n$. Then $h_J$ is a Hessenberg function, and we have $\Hess(N,h_J)\subseteq \Pet{n}$ since the Hessenberg function for $\Pet{n}$ is given by $h(i)=i+1$ for $1\le i\le n-1$ as we saw in Section~\ref{subsect: Hessenberg}.

\begin{example}\label{ex: hJ}
{\rm 
Let $n=10$ and $J= \{1,2\} \sqcup\{4,5,6\} \sqcup\{9\} = J_1\sqcup J_2\sqcup J_3$. Then the configuration of boxes of $h_J$ is given in Figure $\ref{pic: example of hJ}$. Compare the figure with the permutation matrix of $w_J$ in Example~$\ref{eg: wJ}$.}
\begin{figure}[htbp]
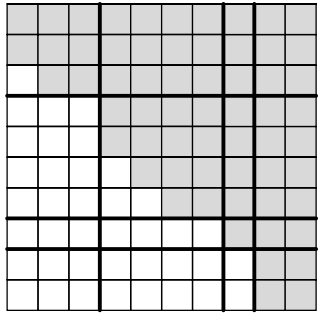

\[
%WinTpicVersion4.32a
{\unitlength 0.1in%
}%
\]
\caption{The Hessenberg function $h_J$.}
\label{pic: example of hJ}
\end{figure}
\end{example}

As we mentioned above, Insko-Tymoczko (\cite{Insko-Tymoczko}) studied $X_J$, and they proved the most part of the following claim. Recall that $X_J$ is defined to be the intersection $X_{w_J}\cap\Pet{n}$ where $X_{w_J}$ is the Schubert variety associated with $w_J$. 

\begin{proposition}\label{prop: XJ and Hess}
For $J\subseteq [n-1]$, we have
\begin{align}\label{eq: two equalities}
 X_J = \overline{X^{\circ}_{w_J}\cap\Pet{n}} = \Hess(N,h_J)
\end{align}
where $N$ is the regular nilpotent matrix given in \eqref{eq: def of N}, and $X^{\circ}_{w_J}$ is the Schubert cell associated with $w_J$. In particular, we have $\dim_{\C} X_J = |J|$.
\end{proposition}

To prove this, we need the following lemma. Let $X^{\circ}_{w}\subseteq Fl_n$ be the Schubert cell associated with $w$ and and $\Omega^{\circ}_w\subseteq Fl_n$ the dual Schubert cell associated with $w$.

\begin{lemma}\label{lem: intersection with dual cell}
The following are equivalent.
\begin{itemize}
 \item[(1)] $X^{\circ}_w\cap \Pet{n}\ne\emptyset$
 \item[(2)] $\Omega^{\circ}_w\cap \Pet{n}\ne\emptyset$
 \item[(3)] $w\in \Pet{n}$ $($i.e. $w=w_J$ for some $J\subseteq[n-1])$
\end{itemize}
\end{lemma}

\begin{proof}
It is clear that (3) implies (1). To see that (1) implies (3), take an  element $z\in X^{\circ}_w\cap \Pet{n}\ne\emptyset$. Since $X^{\circ}_w\cap \Pet{n}\subseteq X^{\circ}_w$ is preserved under the $\C^{\times}$-action on $X^{\circ}_w$, it follows that $t\cdot z\in X^{\circ}_w\cap \Pet{n}$ for all $t\in \C^{\times}$. Noticing that $X^{\circ}_w\cap \Pet{n}\subseteq X^{\circ}_w$ is a closed subset, we have 
\begin{align}\label{eq: limit}
\lim_{t\rightarrow0}t\cdot z\in X^{\circ}_w\cap \Pet{n}.
\end{align}
Under the standard identification $X^{\circ}_w = \C^{\ell(w)}$ (cf.\ \cite[Sect.\ 10.2]{fult97}), the $\C^{\times}$-action on $X^{\circ}_w$ is identified with a linear action with positive weights. Thus it follows that $\lim_{t\rightarrow0}t\cdot z=0$ which corresponds to $w\in X^{\circ}_w$ (cf.\ the proof of Lemma~5 in \cite{Insko-Tymoczko}). Thus it follows that $w\in \Pet{n}$ by \eqref{eq: limit}. The equivalence of (2) and (3) follows by an argument similar to that for the equivalence of (1) and (3).
\end{proof}

\begin{proof}[Proof of Proposition~$\ref{prop: XJ and Hess}$]
Let us first prove that 
\begin{align}\label{eq: first inclusion}
 X^{\circ}_{w_J}\cap\Pet{n} \subseteq \Hess(N,h_J)
 \qquad \text{for each $J\subseteq [n-1]$}.
\end{align}
For this, take an arbitrary element $V_{\bullet}\in X^{\circ}_{w_J}\cap\Pet{n}$. Then we have
\begin{align*}
 NV_i \subseteq V_{i+1} \qquad (1\le i\le n-1).
\end{align*}
To see that $V_{\bullet}\in \Hess(N,h_J)$, we need to show that 
\begin{align}\label{eq: NVp}
 NV_p \subseteq V_{p} \qquad \text{for }p\notin J.
\end{align}
Since we are assuming $V_{\bullet}\in X^{\circ}_{w_J}\cap\Pet{n}$, the flag $V_{\bullet}$ lies in the Schubert cell $X^{\circ}_{w_J}$. Here, the permutation $w_J$ is a product of longest permutations of the symmetric group of smaller ranks as given in \eqref{eq: def of wJ}. Thus it follows (from e.g.\ \cite[Sect.\ 10.2]{fult97}) that 
\begin{align*}
 V_p 
 = 
 \langle e_{1}, e_{2}, \ldots, e_{p}\rangle
 \qquad \text{for }p\notin J ,
\end{align*}
where $e_1,e_2,\ldots,e_n$ are the standard basis of $\C^n$. Since $Ne_1 =0$ and $Ne_i =e_{i-1}$ for $2\le i\le n$, it is clear that \eqref{eq: NVp} follows. Thus we obtain \eqref{eq: first inclusion}.

Now, let us prove the claim \eqref{eq: two equalities} of this proposition. Since we have $X_{w_J}=\bigsqcup_{v\le w_J}  X^{\circ}_{v}$, it follows that
\begin{align}\label{eq: decomp of XJ}
 X_{J}
 = X_{w_J}\cap \Pet{n}
 = \bigsqcup_{v\le w_J} ( X^{\circ}_{v}\cap \Pet{n} )
 = \bigsqcup_{J'\subseteq J} ( X^{\circ}_{w_{J'}}\cap \Pet{n} ), 
\end{align}
where the last equality follows from Lemmas \ref{lem: So-san's lemma} and \ref{lem: intersection with dual cell}. For each intersection $X^{\circ}_{w_{J'}}\cap \Pet{n}$ in the right-most side, we have that $X^{\circ}_{w_{J'}}\cap \Pet{n}\subseteq \Hess(N,h_{J'})$ by \eqref{eq: first inclusion}. 
The condition $J'\subseteq J$ implies that $h_{J'}(i)\le h_J(i)$ for $1\le i\le n$, and hence we have $\Hess(N,h_{J'})\subseteq \Hess(N,h_{J})$. Combining this with the previous inclusion, we see that
\begin{align*}
 X^{\circ}_{w_{J'}}\cap \Pet{n}\subseteq \Hess(N,h_J) 
\end{align*}
in \eqref{eq: decomp of XJ}. Thus it follows that
\begin{align}\label{eq: XJ in HesshJ}
 X_{J} \subseteq \Hess(N,h_J).
\end{align}
Note that $X^{\circ}_{w_J}\cap \Pet{n} \subseteq X_{J}(=X_{w_J}\cap \Pet{n})$ by definition, and hence we have that $\overline{X^{\circ}_{w_J}\cap \Pet{n}} \subseteq X_{J}$ by taking the closure. Combining this with \eqref{eq: XJ in HesshJ}, we obtain that
\begin{align}\label{eq: two inclusions}
 \overline{X^{\circ}_{w_J}\cap \Pet{n}} \subseteq X_{J} \subseteq \Hess(N,h_J).
\end{align}
In this sequence, the both sides have the same dimension.  This is because we have $\dim_{\C} \overline{X^{\circ}_{w_J}\cap\Pet{n}} = |J|$ from \cite[Lemma~9]{Insko-Tymoczko} and $\dim_{\C} \Hess(N,h_J)=|J|$ from \cite[Lemma~7.1]{an-ty}. Since $\Hess(N,h_J)$ is irreducible, the two inclusions in \eqref{eq: two inclusions} are equalities. This completes the proof.
\end{proof}

\vspace{10pt}

Combining Proposition~\ref{prop: XJ and Hess} and a result of Drellich \cite[Theorem 4.5]{Dre1}, we may express $X_J$ as a product of Peterson varieties of smaller ranks as follows. 
Let $J = J_1 \sqcup J_2 \sqcup \cdots \sqcup J_{m}$ be the decomposition into the connected components. Then, by definition, we have $w_J=w_0^{(J_1)} w_0^{(J_2)} \cdots w_0^{(J_{m})}$, where each $w_0^{(J_k)}$ is a product of longest elements in $\mathfrak{S}_{J_k}$ for $1\le k\le m$. Hence it follows that the Schubert variety $X_{w_J}\subseteq Fl_n$ associated with $w_J$ is isomorphic to the product of the flag varieties of smaller ranks:
\begin{align*}
 X_{w_J} = \prod_{k=1}^{m} X_{w_0^{(J_k)}} \cong \prod_{k=1}^{m} Fl_{n_k}, 
\end{align*}
where we set
\begin{align*}
 n_k\coloneqq |J_k|+1 \qquad \text{for $1\le k\le m$}.
\end{align*}
By restricting this isomorphism to $X_J = X_{w_J}\cap \Pet{n}$, it follows from Proposition~\ref{prop: XJ and Hess} and \cite[Theorem 4.5]{Dre1} that $X_{J}$ is isomorphic to a product of Peterson varieties of smaller ranks. 
\begin{corollary}\label{lem: decomp of PetJ}
For $J\subseteq[n-1]$, we have
\begin{align*}
 X_{J} = \prod_{k=1}^{m} X_{J_k} \cong
 \prod_{k=1}^{m} \Pet{n_k},
\end{align*}
where $J = J_1 \sqcup J_2 \sqcup \cdots \sqcup J_{m}$ is the decomposition into the connected components and $n_k= |J_k|+1$ $(1\le k\le m)$.
\end{corollary}

\begin{example}\label{ex: product decomp}
{\rm 
Let $n=10$ and $J =\{1,2\} \sqcup\{4,5,6\} \sqcup\{9\} = J_1\sqcup J_2\sqcup J_3$. The representation matrix of $w_J$ is given in Example~$\ref{eg: wJ}$, and we have 
\begin{align*}
&X_{w_J} \cong Fl_3\times Fl_4\times Fl_2, \\
&X_{J} \cong \Pet{3}\times \Pet{4}\times \Pet{2}.
\end{align*}
}
\end{example}

\smallskip

Compared to Schubert varieties and dual Schubert varieties in $Fl_n$, the structures of $X_J$ and $\Omega_J$ in $\Pet{n}$ are rather simple as we explain below.
To begin with, we make the following definition.
\begin{definition}
For $1\le i\le n-1$, let
\begin{equation}
\begin{split}\label{eq: div 10}
D_{i} \coloneqq X_{[n-1]\setminus\{i\}}
\quad \text{and} \quad
E_{i} \coloneqq \Omega_{\{i\}}.
\end{split}
\end{equation}
where $X_{[n-1]\setminus\{i\}}$ and $\Omega_{\{i\}}$ are defined in \eqref{eq: subavr 10}.
\end{definition}

\begin{lemma}\label{lem: divisors}
For $1\le i\le n-1$, the following hold. 
\begin{itemize}
 \item[(1)] $D_{i}$ and $E_{i}$ have codimension $1$ in $\Pet{n}$.
 \item[(2)] $D_{i}\cap E_{i} = \emptyset$.
\end{itemize}
\end{lemma}

\begin{proof}
For (1), we have $\dim_{\C}D_i=\dim_{\C} X_{[n-1]\setminus\{i\}} =n-2$ by Proposition~\ref{prop: XJ and Hess}. We also have 
\begin{align*}
E_{i} = \Omega_{\{i\}} = \Omega_{w_{\{i\}}}\cap\Pet{n}= \Omega_{s_i}\cap\Pet{n} .
\end{align*}
It is well-known that $\Omega_{s_i}$ is irreducible and it has complex codimension $1$ in $Fl_n$ (\cite[Sect.\ 10.2]{fult97}). Hence $\Omega_{s_i}$ in $Fl_n$ is locally cut out by a single function. We also know that $\Omega_{s_i}\cap \Pet{n}$ is a non-empty proper subset of $\Pet{n}$ since we have $s_i\in \Omega_{s_i}\cap \Pet{n}$ and $\text{id}=w_{\emptyset}\in\Pet{n}\setminus\Omega_{s_i}$. Thus, it follows that $\dim_{\C}E_{i} =n-2$. For (2), the claim follows from Proposition~\ref{prop: point intersection 2}.
\end{proof}

In the next subsection, we will see that $D_i$ and $E_i$ are divisors\footnote{In this paper, a divisor on a variety $Y$ always means the support of an effective Cartier divisor on $Y$, i.e., the zero locus of a section of a line bundle over $Y$.} on $\Pet{n}$. The following claim means that $X_J$ and $\Omega_J$ can be described as intersections of divisors on $\Pet{n}$.

\begin{proposition}\label{prop: divisor expression}
For $J\subseteq [n-1]$, we have
\begin{align}\label{eq: div exp 10}
X_{J} =\bigcap_{i\notin J} D_i
\quad \text{and} \quad
\Omega_{J} = \bigcap_{i\in J} E_i .
\end{align}
\end{proposition}

\begin{proof}
By \eqref{eq: decomp of XJ}, we have
\begin{align*}
 X_{w_{[n-1]\setminus\{i\}}}\cap \Pet{n} = \bigsqcup_{J'\subseteq [n-1]\setminus\{i\}} (X^{\circ}_{w_{J'}}\cap \Pet{n}).
\end{align*}
This implies from the definition of $D_i$ that
\begin{align*}
 \bigcap_{i\notin J} D_i
 = \bigcap_{i\notin J} (X_{w_{[n-1]\setminus\{i\}}}\cap \Pet{n}) 
 = \bigsqcup_{J'\subseteq J} (X^{\circ}_{w_{J'}}\cap \Pet{n}).
\end{align*}
Combining this with \eqref{eq: decomp of XJ}, we obtain the desired claim for $X_J$. An argument similar to this proves that 
\begin{align*}
 \bigcap_{i\in J} E_i
 = \bigcap_{i\in J} (\Omega_{w_{\{i\}}}\cap \Pet{n}) 
 = \bigsqcup_{J''\supseteq J} (\Omega^{\circ}_{w_{J''}}\cap \Pet{n})
 = \Omega_{w_J}\cap\Pet{n}
 = \Omega_J
\end{align*}
by Lemmas \ref{lem: So-san's lemma} and \ref{lem: intersection with dual cell}.
\end{proof}

\begin{example}\label{eg: intersection}
{\rm 
Let $n=9$ and $J=\{2,3,4\}\sqcup\{7,8\}$ so that $[n-1]\setminus J=\{1\}\sqcup\{5,6\}$. Then we have 
\begin{align*}
X_{J} = D_1\cap D_5\cap D_6
\qquad \text{and} \qquad
\Omega_{J} = E_2\cap E_3\cap E_4\cap E_7\cap E_8.
\end{align*}
}
\end{example}

\smallskip

\subsection{Defining equations of $X_J$ and $\Omega_J$}\label{subsec: equation of X and Omega}
Let $B\subseteq {\rm GL}_n(\C)$ be the Borel subgroup of upper-triangular matrices in ${\rm GL}_n(\C)$. Then we have the standard identification $Fl_n={\rm GL}_n(\C)/B$ as is well known. Recall that ${\rm GL}_n(\C)\rightarrow {\rm GL}_n(\C)/B(=Fl_n)$ is a principal $B$-bundle. Thus, for a complex $B$-representation space $V$, we have the associated complex vector bundle\footnote{We take the $B$-action on the product ${\rm GL}_n(\C)\times V$ so that $[g,v]=[gb,b^{-1}\cdot v]$ in the quotient.} over ${\rm GL}_n(\C)/B$:
\begin{align*}
{\rm GL}_n(\C)\times^B V \rightarrow {\rm GL}_n(\C)/B
\quad ; \quad
[g,v]\mapsto gB.
\end{align*}
For a weight $\mu\colon T\rightarrow \C^{\times}$, we obtain $\mu\colon B\rightarrow \C^{\times}$ by composing that with the canonical projection $B\twoheadrightarrow T$, which we also denote by the same symbol. We denote by $\C_{\mu}=\C$ the corresponding 1-dimensional representation space of $B$. Set
\begin{align}\label{eq: def of Lmu 2}
L_{\mu} = {\rm GL}_n(\C)\times^B \C_{\mu}^* = {\rm GL}_n(\C)\times^B \C_{-\mu},
\end{align}
where $\C_{\mu}^*$ is the dual representation space of $\C_{\mu}$. We also denote the restriction $L_{\mu}|_{\Pet{n}}$ by the same symbol $L_{\mu}$ when there are no confusion.

Let us introduce two representations of $B$ associated with each $J\subseteq [n-1]$ as follows. For $1\le i\le n-1$, let $\varpi_i\colon T\rightarrow \C^{\times}$ be the $i$-th fundamental weight of $T$ given by $\text{diag}(t_1,t_2,\ldots,t_n)\mapsto t_1t_2\cdots t_i$. For $J\subseteq [n-1]$, we obtain a representation space of $T$ given by a direct sum
\begin{align*}
\bigoplus_{i\in J} \C_{\varpi_i}^*.
\end{align*}
Through the canonical projection $B\twoheadrightarrow T$, we regard this as a representation of $B$. To introduce the other representation of $B$ associated with $J$, 
let $\alpha_i$ $(1\le i\le n-1)$ be the $i$-th simple root defined as a weight $\alpha_i\colon T\rightarrow \C^{\times}$ given by $\text{diag}(t_1,t_2,\ldots,t_n)\mapsto t_{i}t_{i+1}^{-1}$. Let $H_J\subseteq \mathfrak{gl}_n(\C)$ be the Hessenberg subspace (cf.\ \cite[Sect.~2]{ty}) corresponding to the Hessenberg function $h_J$ defined in \eqref{eq: def of hJ}, that is,
\begin{align}\label{eq: HJ}
&H_{J} \coloneqq \mathfrak{b}\oplus\bigoplus_{i\in J} \mathfrak{g}_{-\alpha_i} \subseteq \mathfrak{gl}_n(\C),
\end{align}
where $\mathfrak{b}=\Lie(B)$ is the Lie algebra of $B$ and each $\mathfrak{g}_{-\alpha_{i}}$ is the standard root space of $\mathfrak{gl}_n(\C)$ associated with the $i$-th negative simple root $-\alpha_i$ $(1\le i\le n-1)$. Since $H_J$ is preserved by the adjoint action of $B$ on $\mathfrak{gl}_n(\C)$, the quotient space
\begin{align*}
H_{[n-1]}/H_J
\end{align*}
is a representation space of $B$. Now, these two representations of $B$ induce the following vector bundles over $Fl_n$: 
\begin{align*}
&U_{J} 
\coloneqq {\rm GL}_n(\C)\times^B \left(H_{[n-1]}/H_J\right), \\
&V_{J} 
\coloneqq {\rm GL}_n(\C)\times^B \left(\bigoplus_{i\in J} \C_{\varpi_i}^*\right).
\end{align*}
If there are no confusion, we denote the restrictions of $U_{J}$ and $V_{J}$ on $\Pet{n}$ by the same symbol. Note that we have
\begin{equation}
\begin{split}\label{eq: rank of U and J}
&\rank U_{J}=(n-1)-|J|, \\
&\rank V_{J}=|J|.
\end{split}
\end{equation}

Recall that 
\begin{align*}
\Pet{n} = \{gB \in {\rm GL}_n(\C)/B \mid g^{-1}Ng\in H_{[n-1]}\}
\end{align*}
(cf.\ \cite{Insko-Tymoczko,Rietsch06} or \cite[Sect.\ 2]{ty}).
Thus, the following map gives a section of $U_{J}$ over $\Pet{n}$ :
\begin{align*}
\phi_J \colon  \Pet{n} \rightarrow U_{J}
\quad ; \quad 
\phi_J(gB) = [g, [g^{-1}Ng]],
\end{align*}
where $[g^{-1}Ng]\in H_{[n-1]}/H_J$ is the class represented by $g^{-1}Ng \in H_{[n-1]}$. For $1\le i\le n-1$, let 
\begin{align*}
\text{det}_i \colon {\rm GL}_n(\C)\rightarrow \C_{\varpi_i}(=\C)
\end{align*} 
be the function which takes the leading principal minor of order $i$. This is a $B$-equivariant function with respect to the multiplication of $B$ on ${\rm GL}_n(\C)$ from the right. Thus, we have a well-defined section 
\begin{align*}
\psi_J \colon  \Pet{n} \rightarrow V_J
\quad ; \quad
gB\mapsto \Big[g,\sum_{i\in J}\text{det}_i(g)\Big].
\end{align*}
The following claim means that $X_J$ and $\Omega_J$ in $\Pet{n}$ are defined by the equation $\phi_J=0$ and $\psi_J=0$, respectively.

\begin{proposition}\label{prop: zero locus}
For $J\subseteq[n-1]$, we have 
\begin{align*}
 X_J=Z(\phi_J) \quad \text{and} \quad \Omega_J =Z(\psi_J) ,
\end{align*}
where $Z(\phi_J)$ and $Z(\psi_J)$ denote the zero loci of the sections $\phi_J$ and $\psi_J$, respectively.
\end{proposition}

\begin{proof}
Since the defining condition of $\Hess(N,h_{J})$ is precisely that $g^{-1}Ng\in H_{J}$ (e.g. \cite[Sect.~2]{ty}), it is clear that we have $X_J=Z(\phi_J)$. It is known that 
\begin{align*}
\Omega_{s_i} = \{gB \in {\rm GL}_n(\C)/B \mid \text{det}_i(g)=0\}
\end{align*}
as subsets of $Fl_n$ (cf.\ \cite[Proposition~9 in Sect.\ 10.6]{fult97}). Thus, it follows that
\begin{align*}
Z(\psi_{\{i\}}) = \Omega_{s_i}\cap\Pet{n}=E_i.
\end{align*}
Now, we obtain that
\begin{align*}
Z(\psi_J) = \bigcap_{i\in J} Z(\psi_{\{i\}}) = \bigcap_{i\in J} E_i = \Omega_{J}
\end{align*}
by Proposition~\ref{prop: divisor expression}.
\end{proof}

\begin{corollary}\label{cor: cartier}
For $1\le i\le n-1$, $D_i$ and $E_i$ are divisors on $\Pet{n}$.
\end{corollary}

\begin{proof}
By \eqref{eq: rank of U and J}, it follows that $U_{[n-1]\setminus\{i\}}$ is a line bundle, and we have $D_i=X_{[n-1]\setminus\{i\}}=Z(\phi_{[n-1]\setminus\{i\}})$ by the previous proposition. This means that $D_i$ is a divisor on $\Pet{n}$. Similarly, $E_i=\Omega_{\{i\}}=Z(\psi_{\{i\}})$ is a divisor on $\Pet{n}$ since $V_{\{i\}}$ is a line bundle by \eqref{eq: rank of U and J}.
\end{proof}

\begin{corollary}\label{cor: codim of OmegaJ}
For $J\subseteq[n-1]$, we have $\codim_{\C} \Omega_J=|J|$ in $\Pet{n}$.
\end{corollary}

\begin{proof}
Recall that $\Omega_J =Z(\psi_J)$ and $\rank V_{J}=|J|$. This means that $\Omega_J$ in $\Pet{n}$ is locally cut out by $|J|$ functions, and it follows that each irreducible component of $\Omega_J$ has codimension at most $|J|$ in $\Pet{n}$ (\cite[Proposition 14.1]{Fulton}). This implies that
\begin{align}\label{eq: ineq from equations}
\dim_{\C}\Omega_J\ge (n-1)-|J|
\end{align}
since $\dim_{\C}\Pet{n}=n-1$ and $\Omega_J=\Omega_{w_J}\cap\Pet{n}\ne\emptyset$. We show that the equality holds by induction on $(n-1)-|J|$. When $J=[n-1]$, we have $\Omega_{[n-1]}=\Omega_{w_0}\cap \Pet{n}=\{w_0\}$ so that the claim is obvious. Let $J\subsetneq[n-1]$, and assume by induction that $\dim_{\C}\Omega_K= (n-1)-|K|$ for $J\subsetneq K\subseteq[n-1]$. We prove that $\dim_{\C}\Omega_J= (n-1)-|J|$. Take an element $i\in[n-1]\setminus J$, and set $K=J\sqcup\{i\}$. Then we have 
\begin{align}\label{eq: decreasing sequence}
\Omega_{K}=\Omega_J\cap E_i \subseteq \Omega_J
\end{align}
by Proposition~\ref{prop: divisor expression}. This means that $\Omega_{K}$ is the zero locus of the section $\psi_{\{i\}}$ of the line bundle $V_{\{i\}}$ restricted over $\Omega_{J}$. Thus we see that 
\begin{align}\label{eq: zero-locus ineq}
\dim_{\C}\Omega_{K} \ge \dim_{\C} \Omega_{J} -1,
\end{align}
which follows by applying \cite[Proposition 14.1]{Fulton} to each irreducible component of $\Omega_{J}$. Namely, the dimension decreases at most by 1 in \eqref{eq: decreasing sequence}. Since we have $\dim_{\C}\Omega_K= (n-1)-|K|$ by the inductive hypothesis, we can rewrite \eqref{eq: zero-locus ineq} as
\begin{align*}
\dim_{\C}\Omega_{J} \le \dim_{\C}\Omega_K +1 = (n-1) - |K| + 1 = (n-1) - |J|.
\end{align*}
Combining this with \eqref{eq: ineq from equations}, we obtain the desired equality.
\end{proof}

\begin{remark}\label{rem: decomposition into line bundles}
The vector bundles $U_J$ and $V_J$ are decomposed into line bundles as
\begin{align*}
U_J = \bigoplus_{i\notin J} L_{\alpha_i}
\qquad \text{and}\qquad
V_J = \bigoplus_{i\in J} L_{\varpi_i}.
\end{align*} 
The first equality follows since we have $H_{[n-1]}/H_J\cong \oplus_{i\notin J}\C_{-\alpha_i}$ as representations of $B$, where the right hand side is a direct sum representation. These decomposition can be thought as the analogue of Proposition $\ref{prop: divisor expression}$ in the language of vector bundles over $\Pet{n}$.
\end{remark}

\bigskip

%%%%%%%%%%%%%%%%%%%%%%%%%%%%%
%%%%%%%%%%%%%%%%%%%%%%%%%%%%%
\section{The cohomology ring of $\Pet{n}$}\label{sec: cohomology}
%%%%%%%%%%%%%%%%%%%%%%%%%%%%%
%%%%%%%%%%%%%%%%%%%%%%%%%%%%%
In this section, we construct an additive basis of the integral cohomology ring $H^*(\Pet{n};\Z)$ by incorporating the geometry established in the previous section.
We also introduce the structure constants of the basis, and provide a geometric proof for their positivity.

For an algebraic variety $X$ which admits a cellular decomposition by complex affines spaces (which is also called a paving by affines), an irreducible Zariski closed subset $Y$ of $X$ has its fundamental cycle (as a reduced scheme) in $H_{2d}(Y;\Z)$, where $d=\dim_{\C}Y$.
By abusing notations, we use the same symbol for its image in $H_{2d}(X;\Z)$ under the induced map $i_*\colon H_{2d}(Y;\Z)\rightarrow H_{2d}(X;\Z)$, where $i$ is the inclusion map $i\colon Z\hookrightarrow X$. See \cite[Appendix B]{fult97} or \cite[Chap.\ 19]{Fulton} for the details.

\vspace{5pt}

\subsection{The homology group of $\Pet{n}$}
Recall that we have a decomposition of $Fl_n$ by the Schubert cells:
\begin{align*}
  Fl_n = \bigsqcup_{w\in\mathfrak{S}_n} X^{\circ}_{w}.
\end{align*}
This induces a decomposition
\begin{align*}
 \Pet{n} = \bigsqcup_{J\subseteq[n-1]} (X^{\circ}_{w_J}\cap \Pet{n}),
\end{align*}
by Lemmas~\ref{lem: So-san's lemma} and \ref{lem: intersection with dual cell}. It is known from \cite{Peterson,ty} that each $X^{\circ}_{w_J}\cap \Pet{n}$ is isomorphic to an affine cell $\C^{|J|}$ and that they form a paving by affines. Recall also from Proposition~\ref{prop: XJ and Hess} that we have
\begin{align*}
 X_J = \overline{X^{\circ}_{w_J}\cap \Pet{n}}
\end{align*}
and that $\dim_{\C}X_J=|J|$ for $J\subseteq[n-1]$. This implies that the cycles represented by $X_J$ for $J\subseteq[n-1]$ form a $\Z$-basis of the homology group $H_*(\Pet{n};\Z)$.

\begin{proposition}\label{prop: homology basis}
$($\cite{Peterson,ty}$)$
For each $J\subseteq [n-1]$, we have $[X_J]\in H_{2|J|}(\Pet{n};\Z)$, and the set
$\{ [X_J] \mid J\subseteq [n-1]\}$ is a $\Z$-basis of $H_*(\Pet{n};\Z)$; 
\begin{align*}
 H_*(\Pet{n};\Z) = \bigoplus_{J\subseteq [n-1]} \Z [X_J] .
\end{align*}
\end{proposition}

\vspace{5pt}

\begin{example}\label{ex: homology basis}
{\rm 
Let $n=4$ so that $[n-1]=\{1,2,3\}$. Then we have 
\begin{align*}
H_*(\Pet{n};\Z) &= \Z [X_{\emptyset}]\oplus(\Z [X_{\{1\}}]\oplus\Z [X_{\{2\}}]\oplus\Z [X_{\{3\}}])\\
&\hspace{75pt}\oplus(\Z [X_{\{1,2\}}]\oplus\Z [X_{\{1,3\}}]\oplus\Z [X_{\{2,3\}}])\oplus\Z [X_{\{1,2,3\}}].
\end{align*}}
\end{example}

\begin{remark}
Compared to $X^{\circ}_{w_J}\cap\Pet{n}(\cong \C^{|J|})$, the geometry of $\Omega^{\circ}_{w_J}\cap \Pet{n}$ is known to encode the quantum cohomology ring of a partial flag variety specified by $J$ $($\cite{Peterson, Rietsch06}$)$.
\end{remark}

\vspace{10pt}

\subsection{The cohomology group of $\Pet{n}$}
For each weight $\mu\colon T\rightarrow \C^{\times}$, we constructed a line bundle $L_{\mu}$ over $Fl_n$ in Section~\ref{subsec: equation of X and Omega}.
Recall also that $\alpha_i$ and $\varpi_i$ are the $i$-th simple root and the $i$-th fundamental weight of $T$ $(1\le i\le n-1)$, respectively.
It is well-known that we have an isomorphism
\begin{align*}
\bigoplus_{i=1}^{n-1} \Z\varpi_i
\stackrel{\cong}{\rightarrow} H^2(Fl_n;\Z) 
\quad ; \quad
\mu \mapsto e(L_{\mu}),
\end{align*}
where we regard each $\mu=a_1\varpi_1+\cdots+a_{n-1}\varpi_{n-1}\ (a_1,\ldots,a_{n-1}\in\Z)$ as the weight $\mu\colon T\rightarrow \C^{\times}$ given by $\text{diag}(t_1,\ldots,t_n)\mapsto t_1^{a_1}(t_1t_2)^{a_2}\cdots(t_1\cdots t_{n-1})^{a_{n-1}}$.
Let $i\colon \Pet{n}\hookrightarrow Fl_n$ be the inclusion map. Insko (\cite{Insko}) proved that the induced homomorphism $i_*\colon H_*(\Pet{n};\Z)\rightarrow H_*(Fl_n;\Z)$ is an injection whose image is a direct summand of $H_*(Fl_n;\Z)$. This means that the map $i_*\colon H_2(\Pet{n};\Z) \rightarrow H_2(Fl_n;\Z)$ on degree $2$ is an isomorphism since $\rank H_2(Fl_n;\Z)=\rank H_2(\Pet{n};\Z) =n-1$ (see Proposition~\ref{prop: homology basis}), and hence the restriction map 
\begin{align*}
i^*\colon H^2(Fl_n;\Z)  \stackrel{\cong}{\rightarrow} H^2(\Pet{n};\Z) 
\end{align*}
on degree $2$ cohomology group is an isomorphism. By combining these isomorphisms, we obtain that $\bigoplus_{i=1}^{n-1} \Z\varpi_i\cong H^2(\Pet{n};\Z)$. In the rest of the paper, we identify $\bigoplus_{i=1}^{n-1} \Z\varpi_i$ with $H^2(\Pet{n};\Z)$ through this isomorphism, and we use the same symbol $\mu\in H^2(\Pet{n};\Z)$ for the element  $e(L_{\mu})$ by abusing notation. For example, we write
\begin{align*}
\alpha_i = e(L_{\alpha_i})
\qquad \text{and} \qquad
\varpi_i = e(L_{\varpi_i})
\end{align*}
as elements in $ H^2(\Pet{n};\Z)$.

In Section~\ref{subsec: equation of X and Omega}, we constructed vector bundles $U_J$ and $V_{J}$ over $\Pet{n}$. Adopting the above notation, we may express the Euler class $e(V_J)$ as a monomial of $\varpi_i$ for each $i\in J$. Namely, for $J\subseteq[n-1]$, we have
\begin{align}\label{eq: monomial exp of e(VJ)}
e(V_J) = \prod_{i\in J} \varpi_i
\end{align}
since the vector bundle $V_{J}$ decomposes into line bundles as follows:
\begin{align*}
V_{J} 
= {\rm GL}_n(\C)\times^B \left(\bigoplus_{i\in J} \C_{\varpi_i}^*\right)
= \bigoplus_{i\in J} L_{\varpi_i}.
\end{align*}

For $J\subseteq[n-1]$, take the decomposition $J = J_1 \sqcup J_2 \sqcup \cdots \sqcup J_{m}$ into the connected components. We set 
\begin{align}\label{eq: def of mJ}
m_J \coloneqq |J_1| ! |J_2| ! \cdots |J_m| !.
\end{align}

\begin{definition} \label{defi:varpi}
For $J\subseteq[n-1]$, let
\begin{align*}
 \varpi_J \coloneqq \frac{1}{m_J} 
 e(V_J)= \frac{1}{m_J} \prod_{i\in J} \varpi_i ,
\end{align*}
where $m_J$ is defined in \eqref{eq: def of mJ}.
\end{definition}

The cohomology class $\varpi_J$ is defined to be an element of $H^{2|J|}(\Pet{n};\Q)$, but we will show that it belongs to the integral cohomology group $H^{2|J|}(\Pet{n};\Z)$.

\begin{example}
{\rm 
Let $n=9$ and $J=\{2,3,4,7,8\}$ so that $J=\{2,3,4\}\sqcup\{7,8\}$ is the decomposition into the connected components. Then we have 
\begin{align*}
 \varpi_J 
 = \frac{1}{3!2!}(\varpi_2\varpi_3\varpi_4)(\varpi_7\varpi_8)
 =\frac{1}{12}\varpi_2\varpi_3\varpi_4\varpi_7\varpi_8 .
\end{align*}
Compare this with Example~$\ref{eg: intersection}$.}
\end{example}

\begin{remark}
We also have 
\begin{align*}
e(U_J) = \prod_{i\notin J} \alpha_i 
\end{align*}
$($cf.\ Remark~$\ref{rem: decomposition into line bundles}$$)$. These decompositions of $e(U_J)$ and $e(V_J)$ can be thought as the cohomological analogue of Proposition $\ref{prop: divisor expression}$. 
\end{remark}

The main purpose of this subsection is to prove that the set of cohomology classes $\{\varpi_J \mid J\subseteq[n-1]\}$ forms a module basis of the integral cohomology group $H^*(\Pet{n};\Z)$. We will state this in Theorem~\ref{thm: Z-basis of cohomology of Peterson}, and we devote the rest of this subsection for its proof.

\begin{lemma}\label{lem: vanishing product}
For $1\le i\le n-1$, we have 
\begin{align*}
\alpha_i \varpi_i = 0 \quad \text{in $H^4(\Pet{n},\Z)$}.
\end{align*}
\end{lemma}

\begin{proof}
Notice that $\alpha_i \varpi_i$ is the Euler class of the rank 2 vector bundle $U_{[n-1]\setminus\{i\}}\oplus V_{\{i\}}=L_{\alpha_i}\oplus L_{\varpi_i}$ (cf.\ Remark~\ref{rem: decomposition into line bundles}). From Section~\ref{subsec: equation of X and Omega}, we have the section $\phi_{[n-1]\setminus\{i\}}+\psi_{\{i\}}$ of this bundle whose zero locus is $Z(\phi_{[n-1]\setminus\{i\}})\cap Z(\psi_{\{i\}})=D_i\cap E_i$ as we saw in the proof of Corollary~\ref{cor: cartier}.
Now, by Lemma~\ref{lem: divisors}~(2), this is the empty set. Thus, $\phi_{[n-1]\setminus\{i\}}+\psi_{\{i\}}$ on $\Pet{n}$ is a nowhere-zero section, and hence the Euler class $\alpha_i \varpi_i$ vanishes (\cite[Property 9.7]{Milnor-Stasheff}).
\end{proof}

\begin{remark}\label{rem: fundamental rel}
In \cite[Corollary~3.4]{fu-ha-ma} and \cite[Theorem 4.1]{ha-ho-ma}, the equations $\alpha_i \varpi_i = 0$ for $1\le i\le n-1$ appeared as the fundamental relations in the presentation of the cohomology ring $H^*(\Pet{n};\C)$. 
\end{remark}

For $1\le i\le n$, let $F_i$ be the $i$-th tautological vector bundle over $Fl_n$ whose fiber at $V_{\bullet}\in Fl_n$ is $V_i$. As a convention, we denote by $F_0$ the trivial sub-bundle of $F_1$ of rank~$0$. The quotient line bundle $L_i\coloneqq F_i/F_{i-1}$ is called the $i$-th tautological line bundle, and we set 
\begin{align} \label{eq:xi}
x_i \coloneqq c_1(L_i^*) \in H^2(Fl_n;\Z)
\quad (1\le i\le n),
\end{align}
where we note that $x_1+\cdots+x_n=0$. We will also denote by the same symbol the restriction of $x_i$ to $H^2(\Pet{n};\Z)$. It is well-known that for $1\le i\le n-1$, we have 
\begin{equation}\label{eq: pi and x}
\begin{split}
&\alpha_i = x_i-x_{i+1}, \\
&\varpi_i = x_1+x_2+\cdots+x_i  .
\end{split}
\end{equation}
For $1\le i<j\le n$, let
\begin{align*}
\alpha_{i,j}\coloneqq \alpha_i+\alpha_{i+1}+\cdots+\alpha_{j-1}=x_i-x_j.
\end{align*}
For a homology cycle $Z\in H_{k}(Fl_n;\Z)$ of degree $k$, the Poincar\'{e} dual of $Z$ is the (unique) cohomology class $\gamma\in H^{2d-k}(Fl_n;\Z)$ $(d=\dim_{\C}Fl_n)$ satisfying $\gamma\cap[Fl_n]=Z$. In the following lemma, we regard $Fl_{n-1}$ as a subvariety of $Fl_n$ whose flags are contained in the linear subspace of $\C^n$ generated by $e_1,e_2,\ldots,e_{n-1}$.
The claim (ii) of the following lemma seems to be well-known, but we provide a proof using Hessenberg varieties for the completeness of the paper.

\begin{lemma}\label{lem: Poincare dual}
The following hold:
\begin{itemize}
 \item[(i)] The Poincar\'e dual of $[\Pet{n}]\in H_*(Fl_n;\Z)$ is $\prod_{j-i\ge2}\alpha_{i, j}\in H^*(Fl_n;\Z)$.
 \item[(ii)] The Poincar\'e dual of $[Fl_{n-1}]\in H_*(Fl_n;\Z)$ is $\frac{1}{n}\alpha_{1,n}\alpha_{2,n}\cdots\alpha_{n-1,n}\in H^*(Fl_n;\Z)$.
\end{itemize}
\end{lemma}

\begin{proof}
We first prove the claim (i). Recall from \eqref{eq: def of hJ} and \eqref{eq: HJ} that $H_J\subseteq \mathfrak{gl}_n(\C)$ for $J\subseteq[n-1]$ is the Hessenberg space corresponding to the Hessenberg function $h_J$. Consider the associated vector bundle 
\begin{align*}
\mathcal{N}\coloneqq {\rm GL}_n(\C)\times^B (\mathfrak{gl}_n(\C)/H_{[n-1]})
\end{align*}
over $Fl_n={\rm GL}_n(\C)/B$. By an argument similar to that used in Section~\ref{subsec: equation of X and Omega}, $\Pet{n}$ can be written as the zero locus of a section of the vector bundle $\mathcal{N}$, and it is shown in \cite[Corollary~3.9]{ab-fu-ze} that the Poincar\'e dual of $[\Pet{n}]$ is the Euler class $e(\mathcal{N})\in H^*(Fl_n;\Z)$. It is straightforward to verify that $e(\mathcal{N})=\prod_{j-i\ge2}\alpha_{i, j}$ by the same inductive argument as that in \cite[Sect.\ 4]{ab-fu-ze} using short exact sequences of vector bundles.

Next we prove the claim (ii). Let $S$ be an $n\times n$ regular semisimple matrix (i.e.\ a diagonal matrix with distinct eigenvalues) and $\Hess(S,h_0)$ a regular semisimple Hessenberg variety, where $h_0$ is a Hessenberg function $h_0\colon [n]\rightarrow [n]$ given by
\begin{align*}
h_0(i) \coloneqq
\begin{cases}
n-1 \quad &(1\le i\le n-1) \\
n &(i=n).
\end{cases}
\end{align*}
It is shown in \cite[Sect,\ 3 and Sect.\ 4]{an-ty} that the Poincar\'{e} dual of $[\Hess(S,h_0)]$ is 
\begin{align*}
(x_1-x_n)(x_2-x_n)\cdots (x_{n-1}-x_n)
\in H^*(Fl_n;\Z),
\end{align*}
where the left hand side is equal to $\alpha_{1,n}\alpha_{2,n}\cdots\alpha_{n-1,n}$ by the  definition of $\alpha_{i,j}$. It is also known that $\Hess(S,h_0)$ has $n$ connected components and that all the connected components give the same cycle $[Fl_{n-1}]$ (\cite[Sect.\ 3]{Teff11}). Thus the Poincar\'e dual of $n[Fl_{n-1}]$ is 
$\alpha_{1,n}\alpha_{2,n}\cdots\alpha_{n-1,n}$, which implies the claim (ii).
\end{proof}

\begin{remark}
{\rm
\cite[Corollary~7.2]{an-ty} with the formula for the double Schubert polynomial in \cite[p.2613]{an-ty} gave a more general formula than that of Lemma~\ref{lem: Poincare dual} (i) for regular nilpotent Hessenberg \textit{schemes} which were not known to be reduced when it was published. After that, \cite{ab-de-ga-ha} proved that they are in fact reduced when they contain $\Pet{n}$ (cf.\ \cite[Remark~3.8]{ab-de-ga-ha}), and the formula are now generalized in \cite{ab-fu-ze} for an arbitrary Lie type.
}
\end{remark}

\vspace{10pt}

For an (irreducible) projective variety $Y$, we denote the fundamental cycle of $Y$ as $[Y]\in H_{2d}(Y;\Z)$, where $d=\dim_{\C} Y$. For a cohomology class $\beta\in H^{2d}(Y;\Z)$, we write
\begin{align*}
\int_Y \beta \coloneqq \langle [Y] , \beta \rangle_{Y} \quad (\in\Z), 
\end{align*} 
where the right hand side is the value of the standard paring
\begin{align*}
\langle \ , \ \rangle_{Y} \colon  H_{2d}(\Pet{n};\Z) \times H^{2d}(\Pet{n};\Z) \rightarrow \Z.
\end{align*}

\begin{proposition}\label{lem: Masuda lemma in fund}
We have
\begin{align*}
\int_{\Pet{n}}\varpi_1\varpi_2\cdots \varpi_{n-1} = (n-1)! .
\end{align*}
\end{proposition}

\begin{proof}
Let us first prove that
\begin{align}\label{eq: basis in fundamental induction}
\int_{\Pet{n}}\varpi_1\varpi_2\cdots \varpi_{n-1} 
= (n-1) \int_{\Pet{n-1}}\varpi_1\varpi_2\cdots \varpi_{n-2}.
\end{align}
For the $\varpi_{n-1}$ in the left hand side of \eqref{eq: basis in fundamental induction}, notice that
\begin{align*}
\varpi_{n-1} = \frac{1}{n}\sum_{i=1}^{n-1} i\alpha_i ,
\end{align*}
which follows from \eqref{eq: pi and x}. Since we have $\varpi_i\alpha_i=0$ for $1\le i\le n-1$ from Lemma~\ref{lem: vanishing product}, we see that
\begin{align}\label{eq: pi->alpha}
\int_{\Pet{n}}\varpi_1\varpi_2\cdots \varpi_{n-1}
= \frac{n-1}{n}\int_{\Pet{n}}\varpi_1\varpi_2\cdots \varpi_{n-2}\alpha_{n-1}.
\end{align}
By Lemma~\ref{lem: Poincare dual}~(i), the right hand side of \eqref{eq: pi->alpha} can be computed as the following integral over $Fl_n$:
\begin{align*}
\frac{n-1}{n}\int_{Fl_n} \varpi_1\varpi_2\cdots \varpi_{n-2}\alpha_{n-1} \prod_{j-i\ge2}\alpha_{i,j}.
\end{align*}
Since we have $\alpha_{n-1}=\alpha_{n-1,n}$, the last expression can be written as
\begin{align*}
\frac{n-1}{n}\int_{Fl_n} \varpi_1\varpi_2\cdots \varpi_{n-2} (\alpha_{1,n}\alpha_{2,n}\cdots\alpha_{n-1,n})\prod_{\substack{j-i\ge2\\j\ne n}}\alpha_{i, j}.
\end{align*}
By Lemma~\ref{lem: Poincare dual}~(ii), we can rewrite this as the following integral over $Fl_{n-1}$:
\begin{align*}
(n-1)\int_{Fl_{n-1}} \varpi_1\varpi_2\cdots \varpi_{n-2} \prod_{\substack{j-i\ge2\\j\ne n}}\alpha_{i, j}.
\end{align*}
Applying Lemma~\ref{lem: Poincare dual}~(i) to $\Pet{n-1}\subseteq Fl_{n-1}$, this can be written as the following integral over $\Pet{n-1}$:
\begin{align*}
(n-1)\int_{\Pet{n-1}} \varpi_1\varpi_2\cdots \varpi_{n-2} .
\end{align*}
Hence we proved \eqref{eq: basis in fundamental induction}.

Now using \eqref{eq: basis in fundamental induction} repeatedly, we obtain that
\begin{align*}
\int_{\Pet{n}}\varpi_1\varpi_2\cdots \varpi_{n-1}
= (n-1)! \int_{\Pet{2}}\varpi_1 .
\end{align*}
Noticing that $\Pet{2}=Fl_2=\P^1$, we see that $\varpi_1(=x_1)$ is the first Chern class of the dual of the standard tautological line bundle over $\P^1$ by \eqref{eq:xi}. Thus the integral in the right hand side is equal to $1$, which completes the proof.
\end{proof}

\vspace{10pt}

\begin{lemma}\label{lem: geometry of mJ}
For $J\subseteq[n-1]$, we have
\begin{align*}
\int_{X_{J}} e(V_J) = m_J , 
\end{align*}
where $m_J=|J_1| ! |J_2| ! \cdots |J_m| !$ is defined in \eqref{eq: def of mJ}.
\end{lemma}

\begin{proof}
The isomorphism given in Corollary~\ref{lem: decomp of PetJ} induces an isomorphism
\begin{align*}
H^*(X_J;\Z) \stackrel{\cong}{\rightarrow} \bigotimes_{k=1}^{m} H^*(\Pet{n_k};\Z)
\end{align*}
which sends
$e(V_J)(=\prod_{i\in J}\varpi_i) \in H^{2|J|}(X_{J};\Z)$ to $\bigotimes_{k=1}^{m} \varpi_1\varpi_2\cdots \varpi_{|J_k|}$. It also induces an isomorphism
\begin{align*}
H_*(X_J;\Z) \stackrel{\cong}{\rightarrow} \bigotimes_{k=1}^{m} H_*(\Pet{n_k};\Z)
\end{align*}
which sends $[X_{J}]\in H_{2|J|}(X_J;\Z)$ to $\bigotimes_{k=1}^{m}[\Pet{n_k}]$. Now the claim follows from Proposition~\ref{lem: Masuda lemma in fund}.
\end{proof}

\vspace{10pt}

\begin{proposition}\label{prop: duality}
For $J,K\subseteq [n-1]$ such that $|J|=|K|$, the degree of the homology class $[X_{J}]$ is the same as the degree of the Euler class $e(V_K)$, 
and we have
\begin{align*} 
\langle [X_{J}] , e(V_K) \rangle_{\Pet{n}} 
=
\begin{cases}
m_J \quad &\text{if $J=K$}, \\
0 &\text{if $J\ne K$}.
\end{cases}
\end{align*}
\end{proposition}

\begin{proof}
Note that we have 
\begin{equation}
\begin{split}\label{eq: XI and omegaJ}
\langle [X_{J}] , e(V_K) \rangle_{\Pet{n}} 
= \int_{X_{J}}  e(V_K).
\end{split}
\end{equation}
For the case $J=K$, the claim follows from the previous lemma. Let us consider the case $J\ne K$. This condition and $|J|= |K|$ imply that $J\not\supseteq K$. 
Recall from Proposition~\ref{prop: zero locus} that we have the section $\psi_{K}\colon \Pet{n} \rightarrow V_{K}$ such that $Z(\psi_K)=\Omega_K$. Thus, Proposition~\ref{prop: point intersection 2} implies that the vector bundle $V_{K}$ restricted on $X_K$ admits a nowhere-zero section (given by $\psi_K$). Thus the Euler class $e(V_K)$ vanishes on $X_J$, and hence the right hand side of \eqref{eq: XI and omegaJ} is equal to $0$ in this case.
\end{proof}

\vspace{10pt}

For $J\subseteq [n-1]$, recall from Definition~\ref{defi:varpi} that
\begin{align*}
\varpi_J 
= \frac{1}{m_J} e(V_J)
= \frac{1}{m_J} \prod_{i\in J}\varpi_i .
\end{align*} 

\begin{theorem}\label{thm: Z-basis of cohomology of Peterson}
For each $J\subseteq [n-1]$, the cohomology class $\varpi_J$ is an element of the integral cohomology $H^{2|J|}(\Pet{n};\Z)$, and the set
\begin{align*}
\{ \varpi_J \in H^*(\Pet{n};\Z) \mid J\subseteq [n-1]\}
\end{align*}
is a $\Z$-basis of $H^*(\Pet{n};\Z)$.
\end{theorem}

\begin{proof}
Recall from Proposition~\ref{prop: homology basis} that $\{[X_{J}] \mid J\subseteq[n-1]\}$ forms a $\Z$-basis of $H_*(\Pet{n};\Z)$. Since the paring between $H_*(\Pet{n};\Z)$ and $H^*(\Pet{n};\Z)$ is perfect, the previous proposition implies the desired claim.
\end{proof}

\vspace{5pt}

\begin{example}
{\rm 
Let $n=4$ so that $[n-1]=\{1,2,3\}$. The additive basis given in Theorem~\ref{thm: Z-basis of cohomology of Peterson} is
\begin{align*}
H^*(\Pet{n};\Z) &= \Z\varpi_{\emptyset}\oplus(\Z\varpi_{\{1\}}\oplus\Z\varpi_{\{2\}}\oplus\Z\varpi_{\{3\}})\\
&\hspace{75pt}\oplus(\Z\varpi_{\{1,2\}}\oplus\Z\varpi_{\{1,3\}}\oplus\Z\varpi_{\{2,3\}})\oplus\Z\varpi_{\{1,2,3\}} .
\end{align*}
As we saw in Proposition $\ref{prop: duality}$, this is the dual basis of the basis of the homology group $H_*(\Pet{n};\Z)$ given in Example~$\ref{ex: homology basis}$.}
\end{example}

\vspace{10pt}

\subsection{Structure constants and their positivity}\label{subsec: positivity}
By Theorem~\ref{thm: Z-basis of cohomology of Peterson}, we can study the cohomology ring $H^*(\Pet{n};\Z)$ in terms of the basis $\{\varpi_J\}_{J\subseteq[n-1]}$. Specifically, we expand the product of two classes $\varpi_J$ and $\varpi_K$ as a linear combination of the basis:
\begin{align} \label{eq:Schubert_calculus_for_varpi_Pet} 
\varpi_J \cdot \varpi_K =\sum_{L \subseteq [n-1]} \d_{JK}^L \, \varpi_L, \ \ \ \d_{JK}^L \in \Z.
\end{align}
The coefficients $\d_{JK}^L$ are called the structure constant for the basis $\{\varpi_J\}_{J\subseteq[n-1]}$. In the following, we explain a geometric interpretation of $\d_{JK}^L$, and deduce  their positivity. Note that the degree of $\varpi_L$ in $H^*(\Pet{n};\Z)$ is $2|L|$ and that the degree of $\varpi_J \cdot \varpi_K$ in $H^*(\Pet{n};\Z)$ is $2(|J|+|K|)$. Thus we may assume that 
\begin{align} \label{eq: degree of piL}
|L|=|J|+|K|
\end{align}
for each summand of \eqref{eq:Schubert_calculus_for_varpi_Pet} since we have $\d_{JK}^L =0$ otherwise. Then by Proposition~\ref{prop: duality}, we have
\begin{align}\label{eq: djkl in terms of pairing}
\d_{JK}^L 
&= \langle [X_{L}] , \varpi_{J} \cdot \varpi_{K} \rangle_{\Pet{n}} 
= \frac{1}{m_J}
\frac{1}{m_K}\int_{X_L} \left(\prod_{j\in J}\varpi_j\right)\left(\prod_{k\in K}\varpi_k\right) , 
\end{align}
where $m_J$ and $m_K$ are the positive integers defined in \eqref{eq: def of mJ}. Now recall that each $\varpi_i\in H^2(X_L;\Z)$ is the Euler class of the line bundle $V_{\{i\}}$ corresponding to the divisor $E_i(=Z(\psi_{\{i\}}))$ on $\Pet{n}$. Hence it follows from \eqref{eq: djkl in terms of pairing} that the structure constant $\d_{JK}^L$ computes an intersection number of (possibly duplicate) divisors $E_i\cap X_L$'s on $X_L$ up to a constant multiple given by $\frac{1}{m_J}\frac{1}{m_K}$ (cf.\ \cite[Sect.\ 1.1.C]{Lazarsfeld}). This provides a geometric interpretation of $\d_{JK}^L$ in \eqref{eq:Schubert_calculus_for_varpi_Pet}, and it leads us to the following positivity.

\begin{proposition}\label{prop: positivity}
We have $\d_{JK}^L \ge0$ for all $J,K,L\subseteq [n-1]$.
\end{proposition}

\begin{proof}
Recall that each line bundle $L_{\varpi_i}$ over $Fl_n$ is nef for $1\le i\le n-1$ (e.g.\ \cite[the proof of Proposition 1.4.1]{Bri} or \cite[Lemma 3.5]{ab-fu-ze2}). Hence the restriction of $L_{\varpi_i}$ over $X_L$ is nef as well for $1\le i\le n-1$. Thus the claim $\d_{JK}^L \ge 0$ follows from \eqref{eq: djkl in terms of pairing} and the positivity of intersection numbers of nef divisors \cite[Example 1.4.16]{Lazarsfeld}.
\end{proof}

\bigskip

%%%%%%%%%%%%%%%%%%%%%%%%%%%%%
%%%%%%%%%%%%%%%%%%%%%%%%%%%%%
\section{Structure constants and Left-right diagrams}\label{sec: LR diagram}
%%%%%%%%%%%%%%%%%%%%%%%%%%%%%
%%%%%%%%%%%%%%%%%%%%%%%%%%%%%
Recall from the previous section that the structure constants $\d_{JK}^L$ are defined to be the coefficients of the expansion formula \eqref{eq:Schubert_calculus_for_varpi_Pet} for the product $\varpi_J \cdot \varpi_K$ for $J,K\subseteq[n-1]$:
\begin{align*}
\varpi_J \cdot \varpi_K =\sum_{L \subseteq [n-1]} \d_{JK}^L \, \varpi_L, \ \ \ \d_{JK}^L \in \Z.
\end{align*}
In this section, we provide a manifestly positive combinatorial formula which computes the structure constant $\d_{JK}^L$ for all $J,K,L\subseteq[n-1]$. We start with the following lemma which tells us how to expand a monomial of $\varpi_1,\ldots,\varpi_{n-1}$ containing a square in the simplest case.

\begin{lemma} \label{lemm:AHKZ}
For $1 \leq \a \leq i \leq \b \leq n-1$, we have
\begin{align*}
\varpi_{i}\cdot 
(\varpi_{\a}\varpi_{\a+1} \cdots \varpi_{\b})
= \frac{\b-i+1}{\b-\a+2} \varpi_{\a-1}\varpi_{\a} \cdots \varpi_{\b} + \frac{i-\a+1}{\b-\a+2} \varpi_{\a} \cdots \varpi_{\b}\varpi_{\b+1}
\end{align*}
in $H^*(\Pet{n};\Z)$, where we take the convention $\varpi_0=\varpi_n=0$.
\end{lemma}

\begin{proof}
We prove the claim by induction on $b-a(\ge0)$. When $b-a=0$, we have $a=i=b$ so that the left hand side is $\varpi_a^2$. Noticing that $\alpha_{a} = -\varpi_{a-1}+2\varpi_{a}-\varpi_{a+1}$ (with the above convention), we have that 
\begin{align*}
\varpi_{a} ( -\varpi_{a-1}+2\varpi_{a}-\varpi_{a+1}) =0
\end{align*}
by Lemma~\ref{lem: vanishing product}. Thus the claim in this case follows since this equality can be expressed as
\begin{align*}
\varpi_{a}^2 = \frac{1}{2}\varpi_{a-1}\varpi_{a}+\frac{1}{2}\varpi_{a}\varpi_{a+1}.
\end{align*}

We now prove the claim for the case $a<b$. Assume by induction that the claim holds for any $a'\le i' \le b'$ with $b'-a'<b-a$. When $i=a$, we have
\begin{align*}
\varpi_{a}(\varpi_{a}\varpi_{a+1}\cdots \varpi_{b})
&=\Big( \varpi_{a}(\varpi_{a}\varpi_{a+1}\cdots \varpi_{b-1}) \Big) \varpi_{b} \\
&=\Big(\frac{b-a}{b-a+1}\varpi_{a-1}\varpi_{a}\cdots \varpi_{b-1} + \frac{1}{b-a+1}\varpi_{a}\varpi_{a+1}\cdots \varpi_{b} \Big)\varpi_{b} \\
&\hspace{185pt}\text{(by the inductive hypothesis)}\\
&=\frac{b-a}{b-a+1} \varpi_{a-1}\varpi_{a}\cdots \varpi_{b} + \frac{1}{b-a+1} \varpi_{a}\Big( \varpi_{a+1}\cdots \varpi_{b-1}\varpi_{b}^2 \Big) \\
&=\frac{b-a}{b-a+1} \varpi_{a-1}\varpi_{a}\cdots \varpi_{b} \\
&\hspace{30pt}+ \frac{1}{(b-a+1)^2} \varpi_{a}\Big( \varpi_{a}\varpi_{a+1}\cdots \varpi_{b} + (b-a) \varpi_{a+1}\varpi_{a+2}\cdots \varpi_{b+1}\Big) \\
&\hspace{185pt} \text{(by the inductive hypothesis)}.
\end{align*}
Since the left hand side and the second summand of the right hand side are proportional, this equation can be written as
\begin{align*}
\frac{(b-a+1)^2-1}{(b-a+1)^2}\varpi_{a}^2(\varpi_{a+1}\cdots \varpi_{b}) 
&=\frac{b-a}{b-a+1} \varpi_{a-1}\varpi_{a}\cdots \varpi_{b} \\
&\hspace{50pt}+ \frac{b-a}{(b-a+1)^2} \varpi_{a}\varpi_{a+1}\cdots \varpi_{b+1}.
\end{align*}
Noticing that $(b-a+1)^2-1=(b-a)(b-a+2)$ for the numerator of the coefficient of the left hand side, we obtain that
\begin{align}\label{eq: relation of p 10}
\varpi_{a}^2(\varpi_{a+1}\cdots \varpi_{b}) = \ \frac{b-a+1}{b-a+2}\varpi_{a-1}\varpi_{a}\cdots \varpi_{b} + \frac{1}{b-a+2} \varpi_{a}\varpi_{a+1}\cdots \varpi_{b+1} 
\end{align}
which verifies the claim for the case $i=a$. Now suppose that $a<i(\le b)$. We then have that 
\begin{align*}
&\varpi_i(\varpi_{a}\varpi_{a+1}\cdots \varpi_{b})\\
&\hspace{20pt}=\varpi_{a}\Big(\varpi_i(\varpi_{a+1}\cdots \varpi_{b})\Big) \\
&\hspace{20pt}=\varpi_{a}\Big( \frac{b-i+1}{b-a+1}\varpi_{a}\varpi_{a+1}\cdots \varpi_{b} + \frac{i-a}{b-a+1} \varpi_{a+1}\varpi_{a+2}\cdots \varpi_{b+1} \Big)\\
&\hspace{220pt} \text{(by the induction hypothesis)} \\
&\hspace{20pt}=\frac{b-i+1}{b-a+1}\varpi_{a}^2(\varpi_{a+1}\cdots \varpi_{b}) + \frac{i-a}{b-a+1}\varpi_{a}\varpi_{a+1}\cdots \varpi_{b+1} \\ 
&\hspace{20pt}=\frac{b-i+1}{b-a+1} \Big( \frac{b-a+1}{b-a+2}\varpi_{a-1}\varpi_{a}\cdots \varpi_{b} + \frac{1}{b-a+2}\varpi_{a}\varpi_{a+1}\cdots \varpi_{b+1}\Big) \\
&\hspace{100pt} + \frac{i-a}{b-a+1}\varpi_{a}\varpi_{a+1}\cdots \varpi_{b+1}\qquad \text{(by \eqref{eq: relation of p 10})}\\
&\hspace{20pt}=\frac{b-i+1}{b-a+2} \varpi_{a-1}\varpi_{a}\cdots \varpi_{b} + \frac{i-a+1}{b-a+2}\varpi_{a}\varpi_{a+1}\cdots \varpi_{b+1}.
\end{align*}
Thus we complete the proof by induction.
\end{proof}

\vspace{10pt}

Lemma~\ref{lemm:AHKZ} is the simplest case of expansions, but it turns out that it provides an effective way for computing the expansion of $\varpi_J\cdot\varpi_K$ for $J,K\subseteq[n-1]$ as we see in the following example.

\vspace{5pt}

\begin{example} \label{ex:calculus}
{\rm Let $n=10$, and take $J=\{1,3,5,6,7 \}$ and $K=\{3,6,8 \}$. 
The product $\varpi_J \cdot \varpi_K$ can be computed by using Lemma~$\ref{lemm:AHKZ}$ repeatedly as follows. We first extract $\varpi_i$'s which produce squares: 
\begin{align*}
\varpi_J \cdot \varpi_K &= \left(\frac{1}{1!1!3!} \varpi_1\varpi_3\varpi_5\varpi_6\varpi_7 \right) \cdot \left(\frac{1}{1!1!1!} \varpi_3\varpi_6\varpi_8 \right) \\
&=\frac{1}{3!} \varpi_6 \cdot \varpi_3 \cdot \left( \varpi_1\varpi_3\varpi_5\varpi_6\varpi_7\varpi_8 \right).
\end{align*}
By applying Lemma~$\ref{lemm:AHKZ}$ to $\varpi_3^2$, this can be computed as
\begin{align*}
&\frac{1}{3!} \varpi_6 \cdot \varpi_3 \cdot \left( \varpi_1\varpi_3\varpi_5\varpi_6\varpi_7\varpi_8 \right) \\
&\hspace{20pt}=\frac{1}{3!}\varpi_6\cdot\left( \varpi_1 \left( \frac{1}{2}\varpi_2\varpi_3 + \frac{1}{2}\varpi_3\varpi_4 \right) \varpi_5\varpi_6\varpi_7\varpi_8 \right) \\
&\hspace{20pt}=\frac{1}{3!} \cdot \frac{1}{2}\varpi_6\cdot\left( \varpi_1\varpi_2\varpi_3\varpi_5\varpi_6\varpi_7\varpi_8 \right)+\frac{1}{3!} \cdot \frac{1}{2}\varpi_6\cdot\left( \varpi_1\varpi_3\varpi_4\varpi_5\varpi_6\varpi_7\varpi_8 \right) .
\end{align*}
Now by applying Lemma~$\ref{lemm:AHKZ}$ to $\varpi_5\varpi_6^2\varpi_7\varpi_8$ in the first summand and $\varpi_3\varpi_4\varpi_5\varpi_6^2\varpi_7\varpi_8$ in the second summand, we can continue our computation as
\begin{align*}
&\frac{1}{3!} \cdot \frac{1}{2}\varpi_6\cdot\left( \varpi_1\varpi_2\varpi_3\varpi_5\varpi_6\varpi_7\varpi_8 \right)+\frac{1}{3!} \cdot \frac{1}{2}\varpi_6\cdot\left( \varpi_1\varpi_3\varpi_4\varpi_5\varpi_6\varpi_7\varpi_8 \right) \\
&\hspace{20pt}=\frac{1}{3!} \cdot \frac{1}{2}\left( \varpi_1\varpi_2\varpi_3 \left( \frac{3}{5}\varpi_4\varpi_5\varpi_6\varpi_7\varpi_8 + \frac{2}{5}\varpi_5\varpi_6\varpi_7\varpi_8\varpi_9 \right) \right)\\
&\hspace{50pt}+\frac{1}{3!} \cdot \frac{1}{2}\left( \varpi_1 \left( \frac{3}{7}\varpi_2\varpi_3\varpi_4\varpi_5\varpi_6\varpi_7\varpi_8 + \frac{4}{7}\varpi_3\varpi_4\varpi_5\varpi_6\varpi_7\varpi_8\varpi_9 \right) \right) \\
&\hspace{20pt}=\frac{1}{3!} \cdot \left( \frac{1}{2} \cdot \frac{3}{5}+ \frac{1}{2} \cdot \frac{3}{7} \right) \cdot 8! \, \varpi_{\{1,2,3,4,5,6,7,8\}}\\
&\hspace{50pt}+\frac{1}{3!} \cdot \frac{1}{2} \cdot \frac{2}{5} \cdot 3! \cdot 5! \, \varpi_{\{1,2,3,5,6,7,8,9\}}
+\frac{1}{3!} \cdot \frac{1}{2} \cdot \frac{4}{7} \cdot 7! \, \varpi_{\{1,3,4,5,6,7,8,9\}} \\
&\hspace{20pt}=3456 \varpi_{\{1,2,3,4,5,6,7,8\}} + 24 \varpi_{\{1,2,3,5,6,7,8,9\}}+240 \varpi_{\{1,3,4,5,6,7,8,9\}}.
\end{align*}
Thus we conclude that 
\begin{align*}
\varpi_J\cdot \varpi_K=3456\varpi_{\{1,2,3,4,5,6,7,8\}}+24\varpi_{\{1,2,3,5,6,7,8,9\}}+240\varpi_{\{1,3,4,5,6,7,8,9\}}
\end{align*}
which gives a particular case of the expansion \eqref{eq:Schubert_calculus_for_varpi_Pet}.
As one can see, the geometric idea behind this computation is the realization of $\Omega_{J}$ by intersecting the divisors $E_i$; see \eqref{eq: div exp 10} and \eqref{eq: monomial exp of e(VJ)}. }
\end{example}

\vspace{10pt}

Let $J,K,L$ be subsets of $[n-1]$. By tracking the computations in the above example, 
it is straightforward to see that if $\varpi_L$ appears in the expansion of the product $\varpi_J \cdot \varpi_K$, then $L$ must contain $J \cup K$. Combining this with \eqref{eq: degree of piL}, we see that
\begin{align}\label{eq: 0 unless}
\d_{JK}^L = 0 \ \textrm{unless} \ L \supseteq J \cup K \textrm{ and } |L|=|J|+|K|.
\end{align}
We now introduce a combinatorial object which effectively computes the structure constants $\d_{JK}^L$.
Because of \eqref{eq: 0 unless}, we always assume that $L \supseteq J \cup K$ and $|L|=|J|+|K|$ in what follows. 
We first prepare the following two steps.
\begin{enumerate}
\item Write the elements of $[n-1]$ in increasing order, and draw a square grid of size $(1+|J \cap K|) \times |L|$ over the subset $L\subseteq[n-1]$. 
On the left side of the grid, write the elements of $J \cap K$ in increasing order from the second row to the bottom row.
\end{enumerate}
For each box in the grid, we define the \textit{row number} of the box as the number which is written beside the row containing the box, and define the \textit{column number} of the box as the number which is written below the column containing the box.
\begin{enumerate}
\setcounter{enumi}{1}
\item Shade the boxes in the first row whose column numbers belong to $J \cup K(\subseteq L)$. Mark each box with a cross $\times$ whose row number is the same as the column number.
\end{enumerate}

\vspace{10pt}

\begin{example} \label{ex:diagram_setup} 
{\rm
Let $n=10$ and take $J=\{1,3,5,6,7 \}$ and $K=\{3,6,8 \}$ as in the previous example. We depict the resulting grids after the steps (1) and (2) for the following two choices of $L$.}
\begin{enumerate}
\item[(i)] {\rm If $L=\{1,2,3,4,5,6,7,8 \}$, then the resulting grid is as follows.}\vspace{10pt}
\begin{center}
%WinTpicVersion4.32a
{\unitlength 0.1in%
\begin{picture}(16.7000,7.4000)(30.0000,-28.7000)%
% BOX 0 0 1 0 Black Black  
% 2 3159 2200 3318 2359
% 
\special{pn 0}%
\special{sh 0.200}%
\special{pa 3159 2200}%
\special{pa 3318 2200}%
\special{pa 3318 2359}%
\special{pa 3159 2359}%
\special{pa 3159 2200}%
\special{ip}%
\special{pn 20}%
\special{pa 3159 2200}%
\special{pa 3318 2200}%
\special{pa 3318 2359}%
\special{pa 3159 2359}%
\special{pa 3159 2200}%
\special{pa 3318 2200}%
\special{fp}%
% BOX 0 0 3 0 Black White  
% 2 3318 2200 3477 2359
% 
\special{pn 20}%
\special{pa 3318 2200}%
\special{pa 3477 2200}%
\special{pa 3477 2359}%
\special{pa 3318 2359}%
\special{pa 3318 2200}%
\special{pa 3477 2200}%
\special{fp}%
% BOX 0 0 1 0 Black Black  
% 2 3477 2200 3636 2359
% 
\special{pn 0}%
\special{sh 0.200}%
\special{pa 3477 2200}%
\special{pa 3636 2200}%
\special{pa 3636 2359}%
\special{pa 3477 2359}%
\special{pa 3477 2200}%
\special{ip}%
\special{pn 20}%
\special{pa 3477 2200}%
\special{pa 3636 2200}%
\special{pa 3636 2359}%
\special{pa 3477 2359}%
\special{pa 3477 2200}%
\special{pa 3636 2200}%
\special{fp}%
% BOX 0 0 3 0 Black White  
% 2 3636 2200 3795 2359
% 
\special{pn 20}%
\special{pa 3636 2200}%
\special{pa 3795 2200}%
\special{pa 3795 2359}%
\special{pa 3636 2359}%
\special{pa 3636 2200}%
\special{pa 3795 2200}%
\special{fp}%
% BOX 0 0 1 0 Black Black  
% 2 3795 2200 3954 2359
% 
\special{pn 0}%
\special{sh 0.200}%
\special{pa 3795 2200}%
\special{pa 3954 2200}%
\special{pa 3954 2359}%
\special{pa 3795 2359}%
\special{pa 3795 2200}%
\special{ip}%
\special{pn 20}%
\special{pa 3795 2200}%
\special{pa 3954 2200}%
\special{pa 3954 2359}%
\special{pa 3795 2359}%
\special{pa 3795 2200}%
\special{pa 3954 2200}%
\special{fp}%
% BOX 0 0 1 0 Black Black  
% 2 3954 2200 4113 2359
% 
\special{pn 0}%
\special{sh 0.200}%
\special{pa 3954 2200}%
\special{pa 4113 2200}%
\special{pa 4113 2359}%
\special{pa 3954 2359}%
\special{pa 3954 2200}%
\special{ip}%
\special{pn 20}%
\special{pa 3954 2200}%
\special{pa 4113 2200}%
\special{pa 4113 2359}%
\special{pa 3954 2359}%
\special{pa 3954 2200}%
\special{pa 4113 2200}%
\special{fp}%
% BOX 0 0 1 0 Black Black  
% 2 4113 2200 4272 2359
% 
\special{pn 0}%
\special{sh 0.200}%
\special{pa 4113 2200}%
\special{pa 4272 2200}%
\special{pa 4272 2359}%
\special{pa 4113 2359}%
\special{pa 4113 2200}%
\special{ip}%
\special{pn 20}%
\special{pa 4113 2200}%
\special{pa 4272 2200}%
\special{pa 4272 2359}%
\special{pa 4113 2359}%
\special{pa 4113 2200}%
\special{pa 4272 2200}%
\special{fp}%
% BOX 0 0 1 0 Black Black  
% 2 4272 2200 4431 2359
% 
\special{pn 0}%
\special{sh 0.200}%
\special{pa 4272 2200}%
\special{pa 4431 2200}%
\special{pa 4431 2359}%
\special{pa 4272 2359}%
\special{pa 4272 2200}%
\special{ip}%
\special{pn 20}%
\special{pa 4272 2200}%
\special{pa 4431 2200}%
\special{pa 4431 2359}%
\special{pa 4272 2359}%
\special{pa 4272 2200}%
\special{pa 4431 2200}%
\special{fp}%
% BOX 0 0 3 0 Black White  
% 2 3159 2359 3318 2518
% 
\special{pn 20}%
\special{pa 3159 2359}%
\special{pa 3318 2359}%
\special{pa 3318 2518}%
\special{pa 3159 2518}%
\special{pa 3159 2359}%
\special{pa 3318 2359}%
\special{fp}%
% BOX 0 0 3 0 Black White  
% 2 3318 2359 3477 2518
% 
\special{pn 20}%
\special{pa 3318 2359}%
\special{pa 3477 2359}%
\special{pa 3477 2518}%
\special{pa 3318 2518}%
\special{pa 3318 2359}%
\special{pa 3477 2359}%
\special{fp}%
% BOX 0 0 3 0 Black White  
% 2 3477 2359 3636 2518
% 
\special{pn 20}%
\special{pa 3477 2359}%
\special{pa 3636 2359}%
\special{pa 3636 2518}%
\special{pa 3477 2518}%
\special{pa 3477 2359}%
\special{pa 3636 2359}%
\special{fp}%
% BOX 0 0 3 0 Black White  
% 2 3636 2359 3795 2518
% 
\special{pn 20}%
\special{pa 3636 2359}%
\special{pa 3795 2359}%
\special{pa 3795 2518}%
\special{pa 3636 2518}%
\special{pa 3636 2359}%
\special{pa 3795 2359}%
\special{fp}%
% BOX 0 0 3 0 Black White  
% 2 3795 2359 3954 2518
% 
\special{pn 20}%
\special{pa 3795 2359}%
\special{pa 3954 2359}%
\special{pa 3954 2518}%
\special{pa 3795 2518}%
\special{pa 3795 2359}%
\special{pa 3954 2359}%
\special{fp}%
% BOX 0 0 3 0 Black White  
% 2 3954 2359 4113 2518
% 
\special{pn 20}%
\special{pa 3954 2359}%
\special{pa 4113 2359}%
\special{pa 4113 2518}%
\special{pa 3954 2518}%
\special{pa 3954 2359}%
\special{pa 4113 2359}%
\special{fp}%
% BOX 0 0 3 0 Black White  
% 2 4113 2359 4272 2518
% 
\special{pn 20}%
\special{pa 4113 2359}%
\special{pa 4272 2359}%
\special{pa 4272 2518}%
\special{pa 4113 2518}%
\special{pa 4113 2359}%
\special{pa 4272 2359}%
\special{fp}%
% BOX 0 0 3 0 Black White  
% 2 4272 2359 4431 2518
% 
\special{pn 20}%
\special{pa 4272 2359}%
\special{pa 4431 2359}%
\special{pa 4431 2518}%
\special{pa 4272 2518}%
\special{pa 4272 2359}%
\special{pa 4431 2359}%
\special{fp}%
% BOX 0 0 3 0 Black White  
% 2 3159 2518 3318 2677
% 
\special{pn 20}%
\special{pa 3159 2518}%
\special{pa 3318 2518}%
\special{pa 3318 2677}%
\special{pa 3159 2677}%
\special{pa 3159 2518}%
\special{pa 3318 2518}%
\special{fp}%
% BOX 0 0 3 0 Black White  
% 2 3318 2518 3477 2677
% 
\special{pn 20}%
\special{pa 3318 2518}%
\special{pa 3477 2518}%
\special{pa 3477 2677}%
\special{pa 3318 2677}%
\special{pa 3318 2518}%
\special{pa 3477 2518}%
\special{fp}%
% BOX 0 0 3 0 Black White  
% 2 3477 2518 3636 2677
% 
\special{pn 20}%
\special{pa 3477 2518}%
\special{pa 3636 2518}%
\special{pa 3636 2677}%
\special{pa 3477 2677}%
\special{pa 3477 2518}%
\special{pa 3636 2518}%
\special{fp}%
% BOX 0 0 3 0 Black White  
% 2 3636 2518 3795 2677
% 
\special{pn 20}%
\special{pa 3636 2518}%
\special{pa 3795 2518}%
\special{pa 3795 2677}%
\special{pa 3636 2677}%
\special{pa 3636 2518}%
\special{pa 3795 2518}%
\special{fp}%
% BOX 0 0 3 0 Black White  
% 2 3795 2518 3954 2677
% 
\special{pn 20}%
\special{pa 3795 2518}%
\special{pa 3954 2518}%
\special{pa 3954 2677}%
\special{pa 3795 2677}%
\special{pa 3795 2518}%
\special{pa 3954 2518}%
\special{fp}%
% BOX 0 0 3 0 Black White  
% 2 3954 2518 4113 2677
% 
\special{pn 20}%
\special{pa 3954 2518}%
\special{pa 4113 2518}%
\special{pa 4113 2677}%
\special{pa 3954 2677}%
\special{pa 3954 2518}%
\special{pa 4113 2518}%
\special{fp}%
% BOX 0 0 3 0 Black White  
% 2 4113 2518 4272 2677
% 
\special{pn 20}%
\special{pa 4113 2518}%
\special{pa 4272 2518}%
\special{pa 4272 2677}%
\special{pa 4113 2677}%
\special{pa 4113 2518}%
\special{pa 4272 2518}%
\special{fp}%
% BOX 0 0 3 0 Black White  
% 2 4272 2518 4431 2677
% 
\special{pn 20}%
\special{pa 4272 2518}%
\special{pa 4431 2518}%
\special{pa 4431 2677}%
\special{pa 4272 2677}%
\special{pa 4272 2518}%
\special{pa 4431 2518}%
\special{fp}%
% STR 2 0 3 0 Black White  
% 4 3200 2760 3200 2840 2 0 0 0
% $1$
\put(32.0000,-28.4000){\makebox(0,0)[lb]{$1$}}%
% STR 2 0 3 0 Black White  
% 4 3359 2760 3359 2840 2 0 0 0
% $2$
\put(33.5900,-28.4000){\makebox(0,0)[lb]{$2$}}%
% STR 2 0 3 0 Black White  
% 4 3518 2760 3518 2840 2 0 0 0
% $3$
\put(35.1800,-28.4000){\makebox(0,0)[lb]{$3$}}%
% STR 2 0 3 0 Black White  
% 4 3677 2760 3677 2840 2 0 0 0
% $4$
\put(36.7700,-28.4000){\makebox(0,0)[lb]{$4$}}%
% STR 2 0 3 0 Black White  
% 4 3836 2760 3836 2840 2 0 0 0
% $5$
\put(38.3600,-28.4000){\makebox(0,0)[lb]{$5$}}%
% STR 2 0 3 0 Black White  
% 4 3995 2760 3995 2840 2 0 0 0
% $6$
\put(39.9500,-28.4000){\makebox(0,0)[lb]{$6$}}%
% STR 2 0 3 0 Black White  
% 4 4154 2760 4154 2840 2 0 0 0
% $7$
\put(41.5400,-28.4000){\makebox(0,0)[lb]{$7$}}%
% STR 2 0 3 0 Black White  
% 4 4313 2760 4313 2840 2 0 0 0
% $8$
\put(43.1300,-28.4000){\makebox(0,0)[lb]{$8$}}%
% STR 2 0 3 0 Black White  
% 4 4472 2760 4472 2840 2 0 0 0
% $9$
\put(44.7200,-28.4000){\makebox(0,0)[lb]{$9$}}%
% STR 2 0 3 0 Black White  
% 4 3020 2431 3020 2510 2 0 0 0
% $3$
\put(30.2000,-25.1000){\makebox(0,0)[lb]{$3$}}%
% STR 2 0 3 0 Black White  
% 4 3020 2589 3020 2669 2 0 0 0
% $6$
\put(30.2000,-26.6900){\makebox(0,0)[lb]{$6$}}%
% LINE 2 0 3 0 Black White  
% 4 3520 2400 3600 2480 3600 2400 3520 2480
% 
\special{pn 8}%
\special{pa 3520 2400}%
\special{pa 3600 2480}%
\special{fp}%
\special{pa 3600 2400}%
\special{pa 3520 2480}%
\special{fp}%
% LINE 2 0 3 0 Black White  
% 4 3990 2560 4070 2640 4070 2560 3990 2640
% 
\special{pn 8}%
\special{pa 3990 2560}%
\special{pa 4070 2640}%
\special{fp}%
\special{pa 4070 2560}%
\special{pa 3990 2640}%
\special{fp}%
% BOX 2 5 3 0 Black White  
% 2 3000 2130 4670 2870
% 
\special{pn 8}%
\special{pa 3000 2130}%
\special{pa 4670 2130}%
\special{pa 4670 2870}%
\special{pa 3000 2870}%
\special{pa 3000 2130}%
\special{ip}%
\end{picture}}%
\end{center}
{\rm For example, the row number of the marked box in the second row is $3$.}\\
\item[(ii)] {\rm If $L'=\{1,2,3,5,6,7,8,9 \}$, then the resulting grid is as follows.}\vspace{10pt}
\begin{center}
%WinTpicVersion4.32a
{\unitlength 0.1in%
\begin{picture}(16.7000,7.4000)(54.0000,-28.7000)%
% BOX 0 0 1 0 Black Black  
% 2 5559 2200 5718 2359
% 
\special{pn 0}%
\special{sh 0.200}%
\special{pa 5559 2200}%
\special{pa 5718 2200}%
\special{pa 5718 2359}%
\special{pa 5559 2359}%
\special{pa 5559 2200}%
\special{ip}%
\special{pn 20}%
\special{pa 5559 2200}%
\special{pa 5718 2200}%
\special{pa 5718 2359}%
\special{pa 5559 2359}%
\special{pa 5559 2200}%
\special{pa 5718 2200}%
\special{fp}%
% BOX 0 0 3 0 Black White  
% 2 5718 2200 5877 2359
% 
\special{pn 20}%
\special{pa 5718 2200}%
\special{pa 5877 2200}%
\special{pa 5877 2359}%
\special{pa 5718 2359}%
\special{pa 5718 2200}%
\special{pa 5877 2200}%
\special{fp}%
% BOX 0 0 1 0 Black Black  
% 2 5877 2200 6036 2359
% 
\special{pn 0}%
\special{sh 0.200}%
\special{pa 5877 2200}%
\special{pa 6036 2200}%
\special{pa 6036 2359}%
\special{pa 5877 2359}%
\special{pa 5877 2200}%
\special{ip}%
\special{pn 20}%
\special{pa 5877 2200}%
\special{pa 6036 2200}%
\special{pa 6036 2359}%
\special{pa 5877 2359}%
\special{pa 5877 2200}%
\special{pa 6036 2200}%
\special{fp}%
% BOX 0 0 3 0 Black White  
% 2 6830 2200 6989 2359
% 
\special{pn 20}%
\special{pa 6830 2200}%
\special{pa 6989 2200}%
\special{pa 6989 2359}%
\special{pa 6830 2359}%
\special{pa 6830 2200}%
\special{pa 6989 2200}%
\special{fp}%
% BOX 0 0 1 0 Black Black  
% 2 6195 2200 6354 2359
% 
\special{pn 0}%
\special{sh 0.200}%
\special{pa 6195 2200}%
\special{pa 6354 2200}%
\special{pa 6354 2359}%
\special{pa 6195 2359}%
\special{pa 6195 2200}%
\special{ip}%
\special{pn 20}%
\special{pa 6195 2200}%
\special{pa 6354 2200}%
\special{pa 6354 2359}%
\special{pa 6195 2359}%
\special{pa 6195 2200}%
\special{pa 6354 2200}%
\special{fp}%
% BOX 0 0 1 0 Black Black  
% 2 6354 2200 6513 2359
% 
\special{pn 0}%
\special{sh 0.200}%
\special{pa 6354 2200}%
\special{pa 6513 2200}%
\special{pa 6513 2359}%
\special{pa 6354 2359}%
\special{pa 6354 2200}%
\special{ip}%
\special{pn 20}%
\special{pa 6354 2200}%
\special{pa 6513 2200}%
\special{pa 6513 2359}%
\special{pa 6354 2359}%
\special{pa 6354 2200}%
\special{pa 6513 2200}%
\special{fp}%
% BOX 0 0 1 0 Black Black  
% 2 6513 2200 6672 2359
% 
\special{pn 0}%
\special{sh 0.200}%
\special{pa 6513 2200}%
\special{pa 6672 2200}%
\special{pa 6672 2359}%
\special{pa 6513 2359}%
\special{pa 6513 2200}%
\special{ip}%
\special{pn 20}%
\special{pa 6513 2200}%
\special{pa 6672 2200}%
\special{pa 6672 2359}%
\special{pa 6513 2359}%
\special{pa 6513 2200}%
\special{pa 6672 2200}%
\special{fp}%
% BOX 0 0 1 0 Black Black  
% 2 6672 2200 6831 2359
% 
\special{pn 0}%
\special{sh 0.200}%
\special{pa 6672 2200}%
\special{pa 6831 2200}%
\special{pa 6831 2359}%
\special{pa 6672 2359}%
\special{pa 6672 2200}%
\special{ip}%
\special{pn 20}%
\special{pa 6672 2200}%
\special{pa 6831 2200}%
\special{pa 6831 2359}%
\special{pa 6672 2359}%
\special{pa 6672 2200}%
\special{pa 6831 2200}%
\special{fp}%
% BOX 0 0 3 0 Black White  
% 2 5559 2359 5718 2518
% 
\special{pn 20}%
\special{pa 5559 2359}%
\special{pa 5718 2359}%
\special{pa 5718 2518}%
\special{pa 5559 2518}%
\special{pa 5559 2359}%
\special{pa 5718 2359}%
\special{fp}%
% BOX 0 0 3 0 Black White  
% 2 5718 2359 5877 2518
% 
\special{pn 20}%
\special{pa 5718 2359}%
\special{pa 5877 2359}%
\special{pa 5877 2518}%
\special{pa 5718 2518}%
\special{pa 5718 2359}%
\special{pa 5877 2359}%
\special{fp}%
% BOX 0 0 3 0 Black White  
% 2 5877 2359 6036 2518
% 
\special{pn 20}%
\special{pa 5877 2359}%
\special{pa 6036 2359}%
\special{pa 6036 2518}%
\special{pa 5877 2518}%
\special{pa 5877 2359}%
\special{pa 6036 2359}%
\special{fp}%
% BOX 0 0 3 0 Black White  
% 2 6830 2359 6989 2518
% 
\special{pn 20}%
\special{pa 6830 2359}%
\special{pa 6989 2359}%
\special{pa 6989 2518}%
\special{pa 6830 2518}%
\special{pa 6830 2359}%
\special{pa 6989 2359}%
\special{fp}%
% BOX 0 0 3 0 Black White  
% 2 6195 2359 6354 2518
% 
\special{pn 20}%
\special{pa 6195 2359}%
\special{pa 6354 2359}%
\special{pa 6354 2518}%
\special{pa 6195 2518}%
\special{pa 6195 2359}%
\special{pa 6354 2359}%
\special{fp}%
% BOX 0 0 3 0 Black White  
% 2 6354 2359 6513 2518
% 
\special{pn 20}%
\special{pa 6354 2359}%
\special{pa 6513 2359}%
\special{pa 6513 2518}%
\special{pa 6354 2518}%
\special{pa 6354 2359}%
\special{pa 6513 2359}%
\special{fp}%
% BOX 0 0 3 0 Black White  
% 2 6513 2359 6672 2518
% 
\special{pn 20}%
\special{pa 6513 2359}%
\special{pa 6672 2359}%
\special{pa 6672 2518}%
\special{pa 6513 2518}%
\special{pa 6513 2359}%
\special{pa 6672 2359}%
\special{fp}%
% BOX 0 0 3 0 Black White  
% 2 6672 2359 6831 2518
% 
\special{pn 20}%
\special{pa 6672 2359}%
\special{pa 6831 2359}%
\special{pa 6831 2518}%
\special{pa 6672 2518}%
\special{pa 6672 2359}%
\special{pa 6831 2359}%
\special{fp}%
% BOX 0 0 3 0 Black White  
% 2 5559 2518 5718 2677
% 
\special{pn 20}%
\special{pa 5559 2518}%
\special{pa 5718 2518}%
\special{pa 5718 2677}%
\special{pa 5559 2677}%
\special{pa 5559 2518}%
\special{pa 5718 2518}%
\special{fp}%
% BOX 0 0 3 0 Black White  
% 2 5718 2518 5877 2677
% 
\special{pn 20}%
\special{pa 5718 2518}%
\special{pa 5877 2518}%
\special{pa 5877 2677}%
\special{pa 5718 2677}%
\special{pa 5718 2518}%
\special{pa 5877 2518}%
\special{fp}%
% BOX 0 0 3 0 Black White  
% 2 5877 2518 6036 2677
% 
\special{pn 20}%
\special{pa 5877 2518}%
\special{pa 6036 2518}%
\special{pa 6036 2677}%
\special{pa 5877 2677}%
\special{pa 5877 2518}%
\special{pa 6036 2518}%
\special{fp}%
% BOX 0 0 3 0 Black White  
% 2 6830 2518 6989 2677
% 
\special{pn 20}%
\special{pa 6830 2518}%
\special{pa 6989 2518}%
\special{pa 6989 2677}%
\special{pa 6830 2677}%
\special{pa 6830 2518}%
\special{pa 6989 2518}%
\special{fp}%
% BOX 0 0 3 0 Black White  
% 2 6195 2518 6354 2677
% 
\special{pn 20}%
\special{pa 6195 2518}%
\special{pa 6354 2518}%
\special{pa 6354 2677}%
\special{pa 6195 2677}%
\special{pa 6195 2518}%
\special{pa 6354 2518}%
\special{fp}%
% BOX 0 0 3 0 Black White  
% 2 6354 2518 6513 2677
% 
\special{pn 20}%
\special{pa 6354 2518}%
\special{pa 6513 2518}%
\special{pa 6513 2677}%
\special{pa 6354 2677}%
\special{pa 6354 2518}%
\special{pa 6513 2518}%
\special{fp}%
% BOX 0 0 3 0 Black White  
% 2 6513 2518 6672 2677
% 
\special{pn 20}%
\special{pa 6513 2518}%
\special{pa 6672 2518}%
\special{pa 6672 2677}%
\special{pa 6513 2677}%
\special{pa 6513 2518}%
\special{pa 6672 2518}%
\special{fp}%
% BOX 0 0 3 0 Black White  
% 2 6672 2518 6831 2677
% 
\special{pn 20}%
\special{pa 6672 2518}%
\special{pa 6831 2518}%
\special{pa 6831 2677}%
\special{pa 6672 2677}%
\special{pa 6672 2518}%
\special{pa 6831 2518}%
\special{fp}%
% STR 2 0 3 0 Black White  
% 4 5600 2760 5600 2840 2 0 0 0
% $1$
\put(56.0000,-28.4000){\makebox(0,0)[lb]{$1$}}%
% STR 2 0 3 0 Black White  
% 4 5759 2760 5759 2840 2 0 0 0
% $2$
\put(57.5900,-28.4000){\makebox(0,0)[lb]{$2$}}%
% STR 2 0 3 0 Black White  
% 4 5918 2760 5918 2840 2 0 0 0
% $3$
\put(59.1800,-28.4000){\makebox(0,0)[lb]{$3$}}%
% STR 2 0 3 0 Black White  
% 4 6077 2760 6077 2840 2 0 0 0
% $4$
\put(60.7700,-28.4000){\makebox(0,0)[lb]{$4$}}%
% STR 2 0 3 0 Black White  
% 4 6236 2760 6236 2840 2 0 0 0
% $5$
\put(62.3600,-28.4000){\makebox(0,0)[lb]{$5$}}%
% STR 2 0 3 0 Black White  
% 4 6395 2760 6395 2840 2 0 0 0
% $6$
\put(63.9500,-28.4000){\makebox(0,0)[lb]{$6$}}%
% STR 2 0 3 0 Black White  
% 4 6554 2760 6554 2840 2 0 0 0
% $7$
\put(65.5400,-28.4000){\makebox(0,0)[lb]{$7$}}%
% STR 2 0 3 0 Black White  
% 4 6713 2760 6713 2840 2 0 0 0
% $8$
\put(67.1300,-28.4000){\makebox(0,0)[lb]{$8$}}%
% STR 2 0 3 0 Black White  
% 4 6872 2760 6872 2840 2 0 0 0
% $9$
\put(68.7200,-28.4000){\makebox(0,0)[lb]{$9$}}%
% STR 2 0 3 0 Black White  
% 4 5420 2431 5420 2510 2 0 0 0
% $3$
\put(54.2000,-25.1000){\makebox(0,0)[lb]{$3$}}%
% STR 2 0 3 0 Black White  
% 4 5420 2589 5420 2669 2 0 0 0
% $6$
\put(54.2000,-26.6900){\makebox(0,0)[lb]{$6$}}%
% LINE 2 0 3 0 Black White  
% 4 5920 2400 6000 2480 6000 2400 5920 2480
% 
\special{pn 8}%
\special{pa 5920 2400}%
\special{pa 6000 2480}%
\special{fp}%
\special{pa 6000 2400}%
\special{pa 5920 2480}%
\special{fp}%
% LINE 2 0 3 0 Black White  
% 4 6390 2560 6470 2640 6470 2560 6390 2640
% 
\special{pn 8}%
\special{pa 6390 2560}%
\special{pa 6470 2640}%
\special{fp}%
\special{pa 6470 2560}%
\special{pa 6390 2640}%
\special{fp}%
% BOX 2 5 3 0 Black White  
% 2 5400 2130 7070 2870
% 
\special{pn 8}%
\special{pa 5400 2130}%
\special{pa 7070 2130}%
\special{pa 7070 2870}%
\special{pa 5400 2870}%
\special{pa 5400 2130}%
\special{ip}%
\end{picture}}%
\end{center}
\end{enumerate}
\end{example}

\vspace{10pt}

We now play a combinatorial game on the grid prepared above. Let us explain the rule of the game inductively.\vspace{5pt}

\noindent
(The game):
\begin{enumerate}
\item[] Assume that some boxes in the $i$-th row are shaded ($1\le i< |J\cap K|+1$). Then shade the boxes in the $(i+1)$-th row whose column numbers are the same as those of the shaded boxes in the $i$-th row. If there is a non-shaded box adjacent to the left (L) or the right (R) of the consecutive string of the shaded boxes in the $(i+1)$-th row containing the marked box, then shade one of them darkly. In this case, continue to the next row. If there are no such boxes, then we stop the game. \\
\end{enumerate}

We say that the combinatorial game explained above is \emph{successful} if we can continue the game to the bottom row. We define a \emph{left-right diagram} associated with $(J,K,L)$ as a configuration of boxes on a square grid of size $(1+|J \cap K|) \times |L|$ over $L(\subseteq[n-1])$ which obtained as the resulting configuration of the shaded boxes of a successful game. We denote by $\Delta_{J K}^L$ the set of left-right diagrams associated with $(J,K,L)$. 

\vspace{10pt}

\begin{example} \label{ex:diagram_game} 
{\rm We take the triples $(J,K,L)$ and $(J,K,L')$ given in Example~\ref{ex:diagram_setup}.}
\begin{enumerate}
\item[(i)] {\rm The left-right diagrams associated with $(J,K,L)$ are $P_1$ and $P_2$ in Figure~\ref{pic:left-right game for L}.}
\begin{figure}[h]
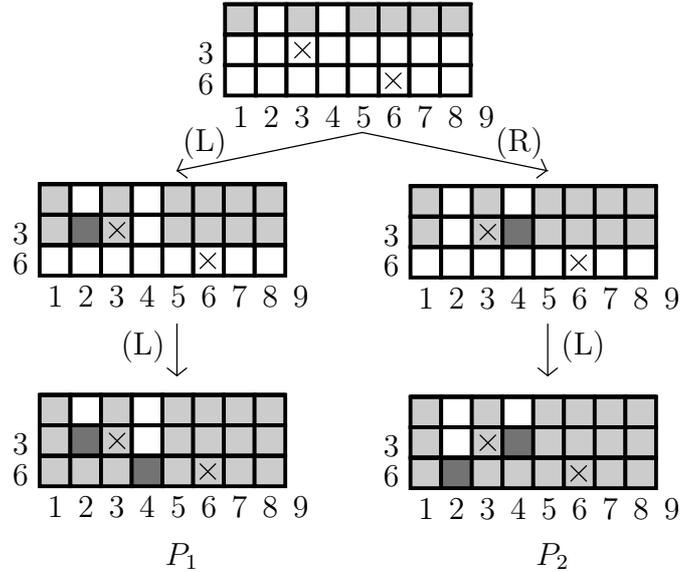

\begin{center}
%WinTpicVersion4.32a
{\unitlength 0.1in%
}%
\end{center}
\vspace{-15pt}
\caption{The games for the $(J,K,L)$.}
\label{pic:left-right game for L}
\end{figure}
\end{enumerate}

\begin{enumerate}
\item[(ii)] {\rm The left-right diagram associated with $(J,K,L')$ is the $P'$ in Figure~\ref{pic:left-right game for L'}.}\\ \vspace{-10pt}

\begin{figure}[h]
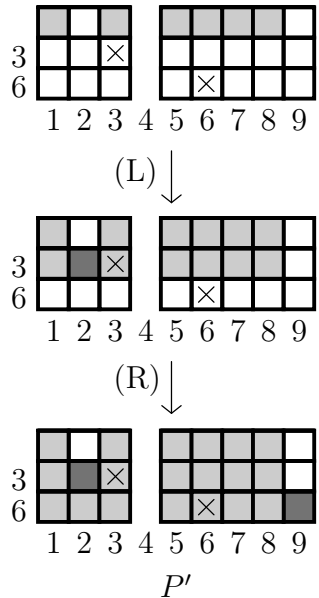

\begin{center}
%WinTpicVersion4.32a
{\unitlength 0.1in%
%
}%
\end{center}
\vspace{-15pt}
\caption{The (unique) game for the $(J,K,L')$.}
\label{pic:left-right game for L'}
\end{figure}
\end{enumerate}
\end{example}

Next, we define the weight of a left-right diagram $P \in \Delta_{J K}^L$ as follows. For each row of $P$ (except for the first row), we consider the consecutive string of the shaded boxes which contains the marked box in the focused row. Then the set of the column numbers for these boxes must be of the form $\{a,a+1,\ldots,b\}$ for some $a,b\in L$, and the column number $i$ of the marked box satisfies $a\le i\le b$. Motivated by Lemma~\ref{lemm:AHKZ}, we assign to this row a positive rational number given by \vspace{10pt}
\begin{equation*}
\begin{split}
&\frac{\b-i+1}{\b-\a+2} \ \ \  \textrm{if the additional box 
%WinTpicVersion4.32a
{\unitlength 0.1in%
\begin{picture}(2.0000,2.0800)(46.9000,-21.5800)%
% BOX 0 0 0 0 Black Black  
% 2 4718 1999 4877 2158
% 
\special{pn 0}%
\special{sh 0.550}%
\special{pa 4718 1999}%
\special{pa 4877 1999}%
\special{pa 4877 2158}%
\special{pa 4718 2158}%
\special{pa 4718 1999}%
\special{ip}%
\special{pn 20}%
\special{pa 4718 1999}%
\special{pa 4877 1999}%
\special{pa 4877 2158}%
\special{pa 4718 2158}%
\special{pa 4718 1999}%
\special{pa 4877 1999}%
\special{fp}%
% BOX 2 5 3 0 Black White  
% 2 4690 1950 4890 2150
% 
\special{pn 8}%
\special{pa 4690 1950}%
\special{pa 4890 1950}%
\special{pa 4890 2150}%
\special{pa 4690 2150}%
\special{pa 4690 1950}%
\special{ip}%
\end{picture}}\ 
is to the left of the marked box
%WinTpicVersion4.32a
{\unitlength 0.1in%
\begin{picture}(2.0000,2.0800)(28.5000,-21.5800)%
% BOX 0 0 1 0 Black Black  
% 2 2877 1999 3036 2158
% 
\special{pn 0}%
\special{sh 0.200}%
\special{pa 2877 1999}%
\special{pa 3036 1999}%
\special{pa 3036 2158}%
\special{pa 2877 2158}%
\special{pa 2877 1999}%
\special{ip}%
\special{pn 20}%
\special{pa 2877 1999}%
\special{pa 3036 1999}%
\special{pa 3036 2158}%
\special{pa 2877 2158}%
\special{pa 2877 1999}%
\special{pa 3036 1999}%
\special{fp}%
% LINE 2 0 3 0 Black White  
% 4 2920 2040 3000 2120 3000 2040 2920 2120
% 
\special{pn 8}%
\special{pa 2920 2040}%
\special{pa 3000 2120}%
\special{fp}%
\special{pa 3000 2040}%
\special{pa 2920 2120}%
\special{fp}%
% BOX 2 5 3 0 Black White  
% 2 2850 1950 3050 2150
% 
\special{pn 8}%
\special{pa 2850 1950}%
\special{pa 3050 1950}%
\special{pa 3050 2150}%
\special{pa 2850 2150}%
\special{pa 2850 1950}%
\special{ip}%
\end{picture}}%
}, \\
&\frac{i-\a+1}{\b-\a+2} \ \ \  \textrm{if the additional box
%WinTpicVersion4.32a
{\unitlength 0.1in%
\begin{picture}(2.0000,2.0800)(46.9000,-21.5800)%
% BOX 0 0 0 0 Black Black  
% 2 4718 1999 4877 2158
% 
\special{pn 0}%
\special{sh 0.550}%
\special{pa 4718 1999}%
\special{pa 4877 1999}%
\special{pa 4877 2158}%
\special{pa 4718 2158}%
\special{pa 4718 1999}%
\special{ip}%
\special{pn 20}%
\special{pa 4718 1999}%
\special{pa 4877 1999}%
\special{pa 4877 2158}%
\special{pa 4718 2158}%
\special{pa 4718 1999}%
\special{pa 4877 1999}%
\special{fp}%
% BOX 2 5 3 0 Black White  
% 2 4690 1950 4890 2150
% 
\special{pn 8}%
\special{pa 4690 1950}%
\special{pa 4890 1950}%
\special{pa 4890 2150}%
\special{pa 4690 2150}%
\special{pa 4690 1950}%
\special{ip}%
\end{picture}}\ 
is to the right of the marked box
%WinTpicVersion4.32a
{\unitlength 0.1in%
\begin{picture}(2.0000,2.0800)(28.5000,-21.5800)%
% BOX 0 0 1 0 Black Black  
% 2 2877 1999 3036 2158
% 
\special{pn 0}%
\special{sh 0.200}%
\special{pa 2877 1999}%
\special{pa 3036 1999}%
\special{pa 3036 2158}%
\special{pa 2877 2158}%
\special{pa 2877 1999}%
\special{ip}%
\special{pn 20}%
\special{pa 2877 1999}%
\special{pa 3036 1999}%
\special{pa 3036 2158}%
\special{pa 2877 2158}%
\special{pa 2877 1999}%
\special{pa 3036 1999}%
\special{fp}%
% LINE 2 0 3 0 Black White  
% 4 2920 2040 3000 2120 3000 2040 2920 2120
% 
\special{pn 8}%
\special{pa 2920 2040}%
\special{pa 3000 2120}%
\special{fp}%
\special{pa 3000 2040}%
\special{pa 2920 2120}%
\special{fp}%
% BOX 2 5 3 0 Black White  
% 2 2850 1950 3050 2150
% 
\special{pn 8}%
\special{pa 2850 1950}%
\special{pa 3050 1950}%
\special{pa 3050 2150}%
\special{pa 2850 2150}%
\special{pa 2850 1950}%
\special{ip}%
\end{picture}}%
}.
\end{split}
\end{equation*}\vspace{5pt}\\
Note that the column number of the additional box is $a-1$ in the former case and is $b+1$ in the latter case (cf. Lemma~\ref{lemm:AHKZ}). We may pictorially interpret this rational number as follows. 
\begin{itemize}
\item The denominator is the number of the shaded boxes counted from the additional box 
%WinTpicVersion4.32a
{\unitlength 0.1in%
\begin{picture}(2.0000,2.0800)(46.9000,-21.5800)%
% BOX 0 0 0 0 Black Black  
% 2 4718 1999 4877 2158
% 
\special{pn 0}%
\special{sh 0.550}%
\special{pa 4718 1999}%
\special{pa 4877 1999}%
\special{pa 4877 2158}%
\special{pa 4718 2158}%
\special{pa 4718 1999}%
\special{ip}%
\special{pn 20}%
\special{pa 4718 1999}%
\special{pa 4877 1999}%
\special{pa 4877 2158}%
\special{pa 4718 2158}%
\special{pa 4718 1999}%
\special{pa 4877 1999}%
\special{fp}%
% BOX 2 5 3 0 Black White  
% 2 4690 1950 4890 2150
% 
\special{pn 8}%
\special{pa 4690 1950}%
\special{pa 4890 1950}%
\special{pa 4890 2150}%
\special{pa 4690 2150}%
\special{pa 4690 1950}%
\special{ip}%
\end{picture}}\ 
to the terminal box lying on the opposite side of the string of shaded boxes across the marked box
%WinTpicVersion4.32a
{\unitlength 0.1in%
\begin{picture}(2.0000,2.0800)(28.5000,-21.5800)%
% BOX 0 0 1 0 Black Black  
% 2 2877 1999 3036 2158
% 
\special{pn 0}%
\special{sh 0.200}%
\special{pa 2877 1999}%
\special{pa 3036 1999}%
\special{pa 3036 2158}%
\special{pa 2877 2158}%
\special{pa 2877 1999}%
\special{ip}%
\special{pn 20}%
\special{pa 2877 1999}%
\special{pa 3036 1999}%
\special{pa 3036 2158}%
\special{pa 2877 2158}%
\special{pa 2877 1999}%
\special{pa 3036 1999}%
\special{fp}%
% LINE 2 0 3 0 Black White  
% 4 2920 2040 3000 2120 3000 2040 2920 2120
% 
\special{pn 8}%
\special{pa 2920 2040}%
\special{pa 3000 2120}%
\special{fp}%
\special{pa 3000 2040}%
\special{pa 2920 2120}%
\special{fp}%
% BOX 2 5 3 0 Black White  
% 2 2850 1950 3050 2150
% 
\special{pn 8}%
\special{pa 2850 1950}%
\special{pa 3050 1950}%
\special{pa 3050 2150}%
\special{pa 2850 2150}%
\special{pa 2850 1950}%
\special{ip}%
\end{picture}}.
\item 
The numerator is the number of the shaded boxes counted from the marked box
%WinTpicVersion4.32a
{\unitlength 0.1in%
\begin{picture}(2.0000,2.0800)(28.5000,-21.5800)%
% BOX 0 0 1 0 Black Black  
% 2 2877 1999 3036 2158
% 
\special{pn 0}%
\special{sh 0.200}%
\special{pa 2877 1999}%
\special{pa 3036 1999}%
\special{pa 3036 2158}%
\special{pa 2877 2158}%
\special{pa 2877 1999}%
\special{ip}%
\special{pn 20}%
\special{pa 2877 1999}%
\special{pa 3036 1999}%
\special{pa 3036 2158}%
\special{pa 2877 2158}%
\special{pa 2877 1999}%
\special{pa 3036 1999}%
\special{fp}%
% LINE 2 0 3 0 Black White  
% 4 2920 2040 3000 2120 3000 2040 2920 2120
% 
\special{pn 8}%
\special{pa 2920 2040}%
\special{pa 3000 2120}%
\special{fp}%
\special{pa 3000 2040}%
\special{pa 2920 2120}%
\special{fp}%
% BOX 2 5 3 0 Black White  
% 2 2850 1950 3050 2150
% 
\special{pn 8}%
\special{pa 2850 1950}%
\special{pa 3050 1950}%
\special{pa 3050 2150}%
\special{pa 2850 2150}%
\special{pa 2850 1950}%
\special{ip}%
\end{picture}}\ 
to the same terminal box as above.
\end{itemize}
\vspace{5pt}
We define the \emph{weight} of $P$ as the product of these positive rational numbers assigned to the rows of $P$ (except for the first row), and denote it by $\wt(P)$.

\vspace{10pt}

\begin{example} \label{ex:diagram_weight}
{\rm Continuing with Example~$\ref{ex:diagram_game}$, 
the weights of the left-right diagrams $P_1, P_2, P'$ can be computed as follows.}
\begin{enumerate}
\item[(i)] {\rm The weights of the left-right diagrams $P_1$ and $P_2$ associated with the $(J,K,L)$ are\vspace{-5pt}
\begin{align*}
\wt(P_1)=\frac{1}{2} \cdot \frac{3}{5} \ \ {\rm and} \ \wt(P_2)=\frac{1}{2} \cdot \frac{3}{7}
\end{align*}
(See Figure~\ref{pic:weights of P}). By construction, these weights appear in the computation of the coefficient of the $\varpi_L=\varpi_{\{1,2,3,4,5,6,7,8\}}$ in Example~\ref{ex:calculus}.}

\begin{figure}[h]
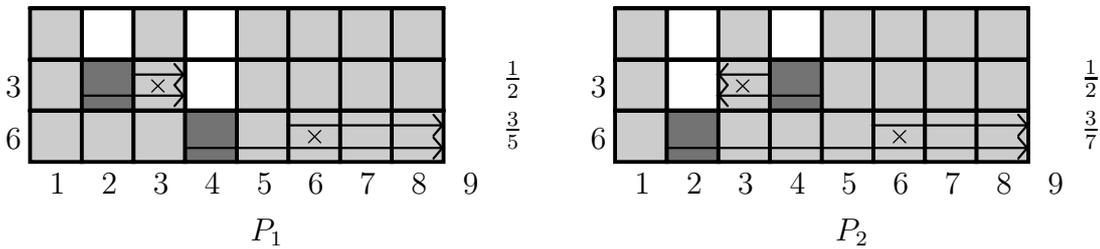

\begin{center}
%WinTpicVersion4.32a
{\unitlength 0.1in%
}%
\end{center}
\vspace{-15pt}
\caption{The computations of the weights of $P_1$ and $P_2$.}
\label{pic:weights of P}
\end{figure}

\item[(ii)] {\rm The weight of the left-right diagram $P' \in \Delta_{JK}^{L'}$ can be computed as
\begin{align*}
\wt(P')=\frac{1}{2} \cdot \frac{2}{5}.
\end{align*}
(See Figure~\ref{pic:weights of P'}).
This weight appears in the computation of the coefficient of the $\varpi_{L'}=\varpi_{\{1,2,3,5,6,7,8,9\}}$ in Example~\ref{ex:calculus} as well.}
\vspace{5pt}
\begin{figure}[h]
\begin{center}
%WinTpicVersion4.32a
{\unitlength 0.1in%
\begin{picture}(31.3800,13.1300)(6.7000,-40.6300)%
% BOX 0 0 1 0 Black Black  
% 2 1043 2796 1311 3064
% 
\special{pn 0}%
\special{sh 0.200}%
\special{pa 1043 2796}%
\special{pa 1311 2796}%
\special{pa 1311 3064}%
\special{pa 1043 3064}%
\special{pa 1043 2796}%
\special{ip}%
\special{pn 20}%
\special{pa 1043 2796}%
\special{pa 1311 2796}%
\special{pa 1311 3064}%
\special{pa 1043 3064}%
\special{pa 1043 2796}%
\special{pa 1311 2796}%
\special{fp}%
% BOX 0 0 3 0 Black White  
% 2 1311 2796 1579 3064
% 
\special{pn 20}%
\special{pa 1311 2796}%
\special{pa 1579 2796}%
\special{pa 1579 3064}%
\special{pa 1311 3064}%
\special{pa 1311 2796}%
\special{pa 1579 2796}%
\special{fp}%
% BOX 0 0 1 0 Black Black  
% 2 1579 2796 1846 3064
% 
\special{pn 0}%
\special{sh 0.200}%
\special{pa 1579 2796}%
\special{pa 1846 2796}%
\special{pa 1846 3064}%
\special{pa 1579 3064}%
\special{pa 1579 2796}%
\special{ip}%
\special{pn 20}%
\special{pa 1579 2796}%
\special{pa 1846 2796}%
\special{pa 1846 3064}%
\special{pa 1579 3064}%
\special{pa 1579 2796}%
\special{pa 1846 2796}%
\special{fp}%
% BOX 0 0 3 0 Black White  
% 2 3186 2796 3454 3064
% 
\special{pn 20}%
\special{pa 3186 2796}%
\special{pa 3454 2796}%
\special{pa 3454 3064}%
\special{pa 3186 3064}%
\special{pa 3186 2796}%
\special{pa 3454 2796}%
\special{fp}%
% BOX 0 0 1 0 Black Black  
% 2 2114 2796 2381 3064
% 
\special{pn 0}%
\special{sh 0.200}%
\special{pa 2114 2796}%
\special{pa 2381 2796}%
\special{pa 2381 3064}%
\special{pa 2114 3064}%
\special{pa 2114 2796}%
\special{ip}%
\special{pn 20}%
\special{pa 2114 2796}%
\special{pa 2381 2796}%
\special{pa 2381 3064}%
\special{pa 2114 3064}%
\special{pa 2114 2796}%
\special{pa 2381 2796}%
\special{fp}%
% BOX 0 0 1 0 Black Black  
% 2 2381 2796 2649 3064
% 
\special{pn 0}%
\special{sh 0.200}%
\special{pa 2381 2796}%
\special{pa 2649 2796}%
\special{pa 2649 3064}%
\special{pa 2381 3064}%
\special{pa 2381 2796}%
\special{ip}%
\special{pn 20}%
\special{pa 2381 2796}%
\special{pa 2649 2796}%
\special{pa 2649 3064}%
\special{pa 2381 3064}%
\special{pa 2381 2796}%
\special{pa 2649 2796}%
\special{fp}%
% BOX 0 0 1 0 Black Black  
% 2 2649 2796 2916 3064
% 
\special{pn 0}%
\special{sh 0.200}%
\special{pa 2649 2796}%
\special{pa 2916 2796}%
\special{pa 2916 3064}%
\special{pa 2649 3064}%
\special{pa 2649 2796}%
\special{ip}%
\special{pn 20}%
\special{pa 2649 2796}%
\special{pa 2916 2796}%
\special{pa 2916 3064}%
\special{pa 2649 3064}%
\special{pa 2649 2796}%
\special{pa 2916 2796}%
\special{fp}%
% BOX 0 0 1 0 Black Black  
% 2 2916 2796 3184 3064
% 
\special{pn 0}%
\special{sh 0.200}%
\special{pa 2916 2796}%
\special{pa 3184 2796}%
\special{pa 3184 3064}%
\special{pa 2916 3064}%
\special{pa 2916 2796}%
\special{ip}%
\special{pn 20}%
\special{pa 2916 2796}%
\special{pa 3184 2796}%
\special{pa 3184 3064}%
\special{pa 2916 3064}%
\special{pa 2916 2796}%
\special{pa 3184 2796}%
\special{fp}%
% BOX 0 0 1 0 Black Black  
% 2 1043 3064 1311 3330
% 
\special{pn 0}%
\special{sh 0.200}%
\special{pa 1043 3064}%
\special{pa 1311 3064}%
\special{pa 1311 3330}%
\special{pa 1043 3330}%
\special{pa 1043 3064}%
\special{ip}%
\special{pn 20}%
\special{pa 1043 3064}%
\special{pa 1311 3064}%
\special{pa 1311 3330}%
\special{pa 1043 3330}%
\special{pa 1043 3064}%
\special{pa 1311 3064}%
\special{fp}%
% BOX 0 0 0 0 Black Black  
% 2 1311 3064 1579 3330
% 
\special{pn 0}%
\special{sh 0.550}%
\special{pa 1311 3064}%
\special{pa 1579 3064}%
\special{pa 1579 3330}%
\special{pa 1311 3330}%
\special{pa 1311 3064}%
\special{ip}%
\special{pn 20}%
\special{pa 1311 3064}%
\special{pa 1579 3064}%
\special{pa 1579 3330}%
\special{pa 1311 3330}%
\special{pa 1311 3064}%
\special{pa 1579 3064}%
\special{fp}%
% BOX 0 0 1 0 Black Black  
% 2 1579 3064 1846 3330
% 
\special{pn 0}%
\special{sh 0.200}%
\special{pa 1579 3064}%
\special{pa 1846 3064}%
\special{pa 1846 3330}%
\special{pa 1579 3330}%
\special{pa 1579 3064}%
\special{ip}%
\special{pn 20}%
\special{pa 1579 3064}%
\special{pa 1846 3064}%
\special{pa 1846 3330}%
\special{pa 1579 3330}%
\special{pa 1579 3064}%
\special{pa 1846 3064}%
\special{fp}%
% BOX 0 0 3 0 Black Black  
% 2 3186 3064 3454 3330
% 
\special{pn 20}%
\special{pa 3186 3064}%
\special{pa 3454 3064}%
\special{pa 3454 3330}%
\special{pa 3186 3330}%
\special{pa 3186 3064}%
\special{pa 3454 3064}%
\special{fp}%
% BOX 0 0 1 0 Black Black  
% 2 2114 3064 2381 3330
% 
\special{pn 0}%
\special{sh 0.200}%
\special{pa 2114 3064}%
\special{pa 2381 3064}%
\special{pa 2381 3330}%
\special{pa 2114 3330}%
\special{pa 2114 3064}%
\special{ip}%
\special{pn 20}%
\special{pa 2114 3064}%
\special{pa 2381 3064}%
\special{pa 2381 3330}%
\special{pa 2114 3330}%
\special{pa 2114 3064}%
\special{pa 2381 3064}%
\special{fp}%
% BOX 0 0 1 0 Black Black  
% 2 2381 3064 2649 3330
% 
\special{pn 0}%
\special{sh 0.200}%
\special{pa 2381 3064}%
\special{pa 2649 3064}%
\special{pa 2649 3330}%
\special{pa 2381 3330}%
\special{pa 2381 3064}%
\special{ip}%
\special{pn 20}%
\special{pa 2381 3064}%
\special{pa 2649 3064}%
\special{pa 2649 3330}%
\special{pa 2381 3330}%
\special{pa 2381 3064}%
\special{pa 2649 3064}%
\special{fp}%
% BOX 0 0 1 0 Black Black  
% 2 2649 3064 2916 3330
% 
\special{pn 0}%
\special{sh 0.200}%
\special{pa 2649 3064}%
\special{pa 2916 3064}%
\special{pa 2916 3330}%
\special{pa 2649 3330}%
\special{pa 2649 3064}%
\special{ip}%
\special{pn 20}%
\special{pa 2649 3064}%
\special{pa 2916 3064}%
\special{pa 2916 3330}%
\special{pa 2649 3330}%
\special{pa 2649 3064}%
\special{pa 2916 3064}%
\special{fp}%
% BOX 0 0 1 0 Black Black  
% 2 2916 3064 3184 3330
% 
\special{pn 0}%
\special{sh 0.200}%
\special{pa 2916 3064}%
\special{pa 3184 3064}%
\special{pa 3184 3330}%
\special{pa 2916 3330}%
\special{pa 2916 3064}%
\special{ip}%
\special{pn 20}%
\special{pa 2916 3064}%
\special{pa 3184 3064}%
\special{pa 3184 3330}%
\special{pa 2916 3330}%
\special{pa 2916 3064}%
\special{pa 3184 3064}%
\special{fp}%
% BOX 0 0 1 0 Black Black  
% 2 1043 3330 1311 3598
% 
\special{pn 0}%
\special{sh 0.200}%
\special{pa 1043 3330}%
\special{pa 1311 3330}%
\special{pa 1311 3598}%
\special{pa 1043 3598}%
\special{pa 1043 3330}%
\special{ip}%
\special{pn 20}%
\special{pa 1043 3330}%
\special{pa 1311 3330}%
\special{pa 1311 3598}%
\special{pa 1043 3598}%
\special{pa 1043 3330}%
\special{pa 1311 3330}%
\special{fp}%
% BOX 0 0 1 0 Black Black  
% 2 1311 3330 1579 3598
% 
\special{pn 0}%
\special{sh 0.200}%
\special{pa 1311 3330}%
\special{pa 1579 3330}%
\special{pa 1579 3598}%
\special{pa 1311 3598}%
\special{pa 1311 3330}%
\special{ip}%
\special{pn 20}%
\special{pa 1311 3330}%
\special{pa 1579 3330}%
\special{pa 1579 3598}%
\special{pa 1311 3598}%
\special{pa 1311 3330}%
\special{pa 1579 3330}%
\special{fp}%
% BOX 0 0 1 0 Black Black  
% 2 1579 3330 1846 3598
% 
\special{pn 0}%
\special{sh 0.200}%
\special{pa 1579 3330}%
\special{pa 1846 3330}%
\special{pa 1846 3598}%
\special{pa 1579 3598}%
\special{pa 1579 3330}%
\special{ip}%
\special{pn 20}%
\special{pa 1579 3330}%
\special{pa 1846 3330}%
\special{pa 1846 3598}%
\special{pa 1579 3598}%
\special{pa 1579 3330}%
\special{pa 1846 3330}%
\special{fp}%
% BOX 0 0 0 0 Black Black  
% 2 3186 3330 3454 3598
% 
\special{pn 0}%
\special{sh 0.550}%
\special{pa 3186 3330}%
\special{pa 3454 3330}%
\special{pa 3454 3598}%
\special{pa 3186 3598}%
\special{pa 3186 3330}%
\special{ip}%
\special{pn 20}%
\special{pa 3186 3330}%
\special{pa 3454 3330}%
\special{pa 3454 3598}%
\special{pa 3186 3598}%
\special{pa 3186 3330}%
\special{pa 3454 3330}%
\special{fp}%
% BOX 0 0 1 0 Black Black  
% 2 2114 3330 2381 3598
% 
\special{pn 0}%
\special{sh 0.200}%
\special{pa 2114 3330}%
\special{pa 2381 3330}%
\special{pa 2381 3598}%
\special{pa 2114 3598}%
\special{pa 2114 3330}%
\special{ip}%
\special{pn 20}%
\special{pa 2114 3330}%
\special{pa 2381 3330}%
\special{pa 2381 3598}%
\special{pa 2114 3598}%
\special{pa 2114 3330}%
\special{pa 2381 3330}%
\special{fp}%
% BOX 0 0 1 0 Black Black  
% 2 2381 3330 2649 3598
% 
\special{pn 0}%
\special{sh 0.200}%
\special{pa 2381 3330}%
\special{pa 2649 3330}%
\special{pa 2649 3598}%
\special{pa 2381 3598}%
\special{pa 2381 3330}%
\special{ip}%
\special{pn 20}%
\special{pa 2381 3330}%
\special{pa 2649 3330}%
\special{pa 2649 3598}%
\special{pa 2381 3598}%
\special{pa 2381 3330}%
\special{pa 2649 3330}%
\special{fp}%
% BOX 0 0 1 0 Black Black  
% 2 2649 3330 2916 3598
% 
\special{pn 0}%
\special{sh 0.200}%
\special{pa 2649 3330}%
\special{pa 2916 3330}%
\special{pa 2916 3598}%
\special{pa 2649 3598}%
\special{pa 2649 3330}%
\special{ip}%
\special{pn 20}%
\special{pa 2649 3330}%
\special{pa 2916 3330}%
\special{pa 2916 3598}%
\special{pa 2649 3598}%
\special{pa 2649 3330}%
\special{pa 2916 3330}%
\special{fp}%
% BOX 0 0 1 0 Black Black  
% 2 2916 3330 3184 3598
% 
\special{pn 0}%
\special{sh 0.200}%
\special{pa 2916 3330}%
\special{pa 3184 3330}%
\special{pa 3184 3598}%
\special{pa 2916 3598}%
\special{pa 2916 3330}%
\special{ip}%
\special{pn 20}%
\special{pa 2916 3330}%
\special{pa 3184 3330}%
\special{pa 3184 3598}%
\special{pa 2916 3598}%
\special{pa 2916 3330}%
\special{pa 3184 3330}%
\special{fp}%
% STR 2 0 3 0 Black White  
% 4 1144 3628 1144 3762 2 0 0 0
% $1$
\put(11.4400,-37.6200){\makebox(0,0)[lb]{$1$}}%
% STR 2 0 3 0 Black White  
% 4 1412 3628 1412 3762 2 0 0 0
% $2$
\put(14.1200,-37.6200){\makebox(0,0)[lb]{$2$}}%
% STR 2 0 3 0 Black White  
% 4 1679 3628 1679 3762 2 0 0 0
% $3$
\put(16.7900,-37.6200){\makebox(0,0)[lb]{$3$}}%
% STR 2 0 3 0 Black White  
% 4 1946 3628 1946 3762 2 0 0 0
% $4$
\put(19.4600,-37.6200){\makebox(0,0)[lb]{$4$}}%
% STR 2 0 3 0 Black White  
% 4 2215 3628 2215 3762 2 0 0 0
% $5$
\put(22.1500,-37.6200){\makebox(0,0)[lb]{$5$}}%
% STR 2 0 3 0 Black White  
% 4 2482 3628 2482 3762 2 0 0 0
% $6$
\put(24.8200,-37.6200){\makebox(0,0)[lb]{$6$}}%
% STR 2 0 3 0 Black White  
% 4 2749 3628 2749 3762 2 0 0 0
% $7$
\put(27.4900,-37.6200){\makebox(0,0)[lb]{$7$}}%
% STR 2 0 3 0 Black White  
% 4 3016 3628 3016 3762 2 0 0 0
% $8$
\put(30.1600,-37.6200){\makebox(0,0)[lb]{$8$}}%
% STR 2 0 3 0 Black White  
% 4 3284 3628 3284 3762 2 0 0 0
% $9$
\put(32.8400,-37.6200){\makebox(0,0)[lb]{$9$}}%
% STR 2 0 3 0 Black White  
% 4 916 3394 916 3527 2 0 0 0
% $6$
\put(9.1600,-35.2700){\makebox(0,0)[lb]{$6$}}%
% STR 2 0 3 0 Black White  
% 4 916 3126 916 3261 2 0 0 0
% $3$
\put(9.1600,-32.6100){\makebox(0,0)[lb]{$3$}}%
% STR 2 0 3 0 Black White  
% 4 3556 3398 3556 3531 2 0 0 0
% $\frac{2}{5}$
\put(35.5600,-35.3100){\makebox(0,0)[lb]{$\frac{2}{5}$}}%
% STR 2 0 3 0 Black White  
% 4 3556 3129 3556 3264 2 0 0 0
% $\frac{1}{2}$
\put(35.5600,-32.6400){\makebox(0,0)[lb]{$\frac{1}{2}$}}%
% STR 2 0 3 0 Black White  
% 4 2183 3903 2183 4036 2 0 0 0
% $P'$
\put(21.8300,-40.3600){\makebox(0,0)[lb]{$P'$}}%
% LINE 2 0 3 0 Black White  
% 2 1673 3170 1737 3233
% 
\special{pn 8}%
\special{pa 1673 3170}%
\special{pa 1737 3233}%
\special{fp}%
% LINE 2 0 3 1 Black White  
% 2 1737 3170 1673 3233
% 
\special{pn 8}%
\special{pa 1737 3170}%
\special{pa 1673 3233}%
\special{fp}%
% LINE 1 0 3 0 Black White  
% 2 1308 3252 1846 3252
% 
\special{pn 13}%
\special{pa 1308 3252}%
\special{pa 1846 3252}%
\special{fp}%
% LINE 1 0 3 0 Black White  
% 2 1572 3142 1837 3142
% 
\special{pn 13}%
\special{pa 1572 3142}%
\special{pa 1837 3142}%
\special{fp}%
% LINE 1 0 3 0 Black White  
% 2 1837 3142 1791 3087
% 
\special{pn 13}%
\special{pa 1837 3142}%
\special{pa 1791 3087}%
\special{fp}%
% LINE 1 0 3 1 Black White  
% 2 1837 3142 1791 3197
% 
\special{pn 13}%
\special{pa 1837 3142}%
\special{pa 1791 3197}%
\special{fp}%
% LINE 1 0 3 0 Black White  
% 2 1837 3252 1791 3197
% 
\special{pn 13}%
\special{pa 1837 3252}%
\special{pa 1791 3197}%
\special{fp}%
% LINE 1 0 3 1 Black White  
% 2 1837 3252 1791 3306
% 
\special{pn 13}%
\special{pa 1837 3252}%
\special{pa 1791 3306}%
\special{fp}%
% LINE 1 0 3 0 Black White  
% 2 2116 3525 3448 3525
% 
\special{pn 13}%
\special{pa 2116 3525}%
\special{pa 3448 3525}%
\special{fp}%
% LINE 1 0 3 0 Black White  
% 2 2118 3525 2164 3471
% 
\special{pn 13}%
\special{pa 2118 3525}%
\special{pa 2164 3471}%
\special{fp}%
% LINE 1 0 3 1 Black White  
% 2 2118 3525 2164 3580
% 
\special{pn 13}%
\special{pa 2118 3525}%
\special{pa 2164 3580}%
\special{fp}%
% LINE 2 0 3 0 Black White  
% 2 2485 3434 2548 3498
% 
\special{pn 8}%
\special{pa 2485 3434}%
\special{pa 2548 3498}%
\special{fp}%
% LINE 2 0 3 1 Black White  
% 2 2548 3434 2485 3498
% 
\special{pn 8}%
\special{pa 2548 3434}%
\special{pa 2485 3498}%
\special{fp}%
% LINE 1 0 3 0 Black White  
% 2 2115 3407 2648 3407
% 
\special{pn 13}%
\special{pa 2115 3407}%
\special{pa 2648 3407}%
\special{fp}%
% LINE 1 0 3 0 Black White  
% 2 2120 3400 2166 3345
% 
\special{pn 13}%
\special{pa 2120 3400}%
\special{pa 2166 3345}%
\special{fp}%
% LINE 1 0 3 1 Black White  
% 2 2120 3400 2166 3454
% 
\special{pn 13}%
\special{pa 2120 3400}%
\special{pa 2166 3454}%
\special{fp}%
% BOX 2 5 3 0 Black White  
% 2 670 2750 3808 4063
% 
\special{pn 8}%
\special{pa 670 2750}%
\special{pa 3808 2750}%
\special{pa 3808 4063}%
\special{pa 670 4063}%
\special{pa 670 2750}%
\special{ip}%
\end{picture}}%
\end{center}
\vspace{-10pt}
\caption{The computation of the weight of $P'$.}
\label{pic:weights of P'}
\end{figure}
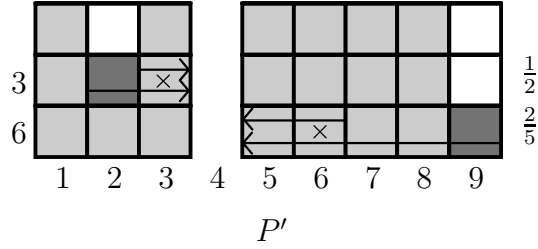
\end{enumerate}

\end{example}

%\vspace{5pt}
We now summarize our computation of the structure constants. For $J\subseteq [n-1]$, recall from Definition~\ref{defi:varpi} that
\begin{align*}
\varpi_J 
= \frac{1}{m_J} e(V_J)
= \frac{1}{m_J} \prod_{i\in J}\varpi_i .
\end{align*}

\begin{theorem} \label{theorem:LeftRight_diagram}
Let $J,K$ be subsets of $[n-1]$. In $H^*(\Pet{n};\Z)$, we have
\begin{align}\label{eq:structure_constants_expansion}
\varpi_J \cdot \varpi_K =\sum_{\substack{K\cup J\subseteq L \subseteq [n-1]\\ |L|=|J|+|K|}} \d_{JK}^L \, \varpi_L, \ \ \ \d_{JK}^L \in \Z,
\end{align}
and the structure constant $\d_{JK}^L$ in this equality is given by
\begin{align*}
\d_{JK}^L=\frac{m_L}{m_Jm_K} \sum_{P \in \Delta_{J K}^L} \wt(P), 
\end{align*}
where $\Delta_{J K}^L$ is the set of left-right diagrams and $\wt(P)$ is the weight of $P$ defined above. In particular, we have $\d_{JK}^L=0$ in $\eqref{eq:structure_constants_expansion}$ if and only if $\Delta_{J K}^L = \emptyset$.
\end{theorem}

\begin{proof}
Recall from Definition~\ref{defi:varpi} that
\begin{align}\label{eq: rational prod}
\varpi_J \cdot \varpi_K = \frac{1}{m_Jm_K} \left( \prod_{j\in J} \varpi_j \right) \cdot \left(\prod_{k\in K} \varpi_k \right) .
\end{align}
If $J\cap K=\emptyset$, then this does not contain a square of $\varpi_1,\ldots,\varpi_{n-1}$, and it is clearly equal to 
\begin{align*}
\frac{m_L}{m_Jm_K} \varpi_L.
\end{align*}
Thus we may assume that $J\cap K\ne\emptyset$ which implies that the right hand side of \eqref{eq: rational prod} contains some squares. By extracting the terms which produce the squares, we can express the product in the right hand side of \eqref{eq: rational prod} as
\begin{align}\label{eq: rational prod rewrite}
\left( \prod_{j\in J} \varpi_j \right) \cdot \left(\prod_{k\in K} \varpi_k \right) 
= \left(\prod_{i\in J \cap K} \varpi_i \right)  \cdot \left( \prod_{q\in J \cup K} \varpi_q \right).
\end{align}
We compute the product in the right hand side of this equality. For this purpose, take the decomposition $J \cup K= M_1 \sqcup \cdots \sqcup M_s$ into the connected components. Let $i$ be the smallest element of $J \cap K$. Then we have $i\in M_r$ for some $r$ ($1\le r\le s$). Since $M_r$ is connected, we can express it as $M_r=\{a,a+1,\ldots,b\}$ for some $a,b\in J \cup K$ with $a\le i\le b$. Then, by Lemma~\ref{lemm:AHKZ} the product $\varpi_i\cdot \left( \prod_{q\in J \cup K} \varpi_q \right)$ can be expanded as 
\begin{align*}
\varpi_i\cdot \left( \prod_{q\in J \cup K} \varpi_q \right)=
\frac{\b-i+1}{\b-\a+2} \prod_{q\in J \cup K \cup \{a-1\} } \varpi_q + \frac{i-\a+1}{\b-\a+2} \prod_{q\in J \cup K \cup \{b+1\} } \varpi_q ,
\end{align*}
where we have no squares of $\varpi_q$'s in the right hand side since $M_r$ is a connected component of $J\cup K$. If $|J\cap K|\ge 2$, then let $i'$ be the smallest element of $J\cap K\setminus\{i\}$. Multiplying $\varpi_{i'}$ to the right hand side of this equality, we can expand it by square-free monomials in $\varpi_1,\ldots,\varpi_{n-1}$ by Lemma~\ref{lemm:AHKZ} again (cf.\ Example~\ref{ex:calculus}). Repeating this procedure for each element of $J \cap K$ in increasing order, we obtain that
\begin{align*}
\left(\prod_{i\in J \cap K} \varpi_i \right) \cdot \left( \prod_{q\in J \cup K} \varpi_q \right) = \sum_{\substack{L \supseteq J \cup K\\ |L|=|J|+|K|}} \left(\Biggl(\sum_{P \in \Delta_{J K}^L} \wt(P)  \Biggl) \prod_{q\in L} \varpi_q \right)
\end{align*}
by the construction of the left-right diagrams and their weights. Combining this with \eqref{eq: rational prod} and \eqref{eq: rational prod rewrite}, we obtain that 
\begin{align*} 
\varpi_J \cdot \varpi_K = \sum_{\substack{L \supseteq J \cup K\\ |L|=|J|+|K|}}  \left(\frac{m_L}{m_Jm_K}\sum_{P \in \Delta_{J K}^L} \wt(P) \right) \varpi_L,
\end{align*}
which implies the desired claim.
\end{proof}

\begin{example} \label{ex:structure_constant_diagram}
{\rm Let $n=10$ and take $J=\{1,3,5,6,7 \}$, $K=\{3,6,8 \}$ as in Example~\ref{ex:calculus}. We compute the coefficients in \eqref{eq:structure_constants_expansion} for the following two choices of $L$. Note that we have $m_J=3!$ and $m_K=1$.}

\begin{enumerate}
\item[(i)] {\rm For $L=\{1,2,3,4,5,6,7,8 \}$, we have $m_L=8!$, and the weights of left-right diagrams associated with the $(J,K,L)$ are computed in Example~$\ref{ex:diagram_weight}$. Hence we obtain that
\begin{align*}
d_{JK}^L=\frac{m_L}{m_Jm_K} \sum_{P \in \Delta_{J K}^L} \wt(P)=\frac{8!}{3!} \left( \frac{1}{2} \cdot \frac{3}{5} + \frac{1}{2} \cdot \frac{3}{7} \right)=3456
\end{align*}
which coincides with the coefficient of $\varpi_L=\varpi_{\{1,2,3,4,5,6,7,8 \}}$ in Example~$\ref{ex:calculus}$.}

\item[(ii)] {\rm For $L'=\{1,2,3,5,6,7,8,9 \}$, we have $m_{L'}=3!\cdot 5!$, and the weight of the left-right diagram associated with the $(J,K,L')$ are computed in Example~$\ref{ex:diagram_weight}$. Hence we obtain that
\begin{align*}
d_{JK}^{L'}=\frac{m_{L'}}{m_Jm_K} \sum_{P \in \Delta_{J K}^{L'}} \wt(P)=\frac{3!\cdot 5!}{3!} \left( \frac{1}{2} \cdot \frac{2}{5} \right)=24
\end{align*}
which coincides with the coefficient of $\varpi_{L'}=\varpi_{\{1,2,3,5,6,7,8,9 \}}$ in Example~$\ref{ex:calculus}$.}
\end{enumerate}
\end{example}

\begin{remark}
{\rm 
Theorem~\ref{theorem:LeftRight_diagram} provides the combinatorial description of the computation demonstrated in Example~$\ref{ex:calculus}$.
As we observed there, the geometric idea behind our computation is the realization of $\Omega_{J}$ by intersecting the divisors $E_i$. 
}
\end{remark}

\bigskip

%%%%%%%%%%%%%%%%%%%%%%%%%%%%%
%%%%%%%%%%%%%%%%%%%%%%%%%%%%%
\section{Relations to other works}
\label{sec: relations to other works}
%%%%%%%%%%%%%%%%%%%%%%%%%%%%%
%%%%%%%%%%%%%%%%%%%%%%%%%%%%%

In this section, we clarify how the results in this paper are related to other works in existing literatures. Especially, we explain the relations to the work of Goldin-Gorbutt (\cite{GoGo}) on \textit{Peterson Schubert calculus} and to the works of Berget-Spink-Tseng (\cite{Berget-Spink-Tseng}), Nadeau-Tewari (\cite{Nadeau-Tewari}), and the second author  (\cite{Horiguchi21}) on \textit{mixed Eulerian numbers}. We emphasize that \cite{Berget-Spink-Tseng, GoGo, Nadeau-Tewari} are announced earlier than this paper.

\subsection{Relations to Peterson Schubert calculus}
We begin with reviewing the motivation of Peterson Schubert calculus from \cite{BaHa, Dre2, GoGo, HaTy}. We first note that these papers studied the equivariant cohomology ring of $\Pet{n}$ with respect to the $\C^{\times}$-action explained in Section~\ref{subsect: combi}, but we focus on the ordinary cohomology ring to compare with our computation (see \cite{BaHa, Dre2, GoGo, HaTy} for the results in the equivariant cohomology). Recall that the dual Schubert variety $\Omega_w$ associated with $w \in \mathfrak{S}_n$ determines the homology cycle $[\Omega_w]$ in $H_*(Fl_n ;\Z)$. We denote by $\sigma_w \in H^{2\ell(w)}(Fl_n ;\Z)$ the Poincar\'e dual of $[\Omega_w]$, which is called the Schubert class associated with $w$. It is well-known that the set of Schubert classes $\{\sigma_w \mid w \in \Sn\}$ forms an additive basis of $H^*(Fl_n ;\Z)$. Thus we may express the product $\sigma_u \cdot \sigma_v$ as a linear combination of the Schubert classes:
\begin{align*} 
\sigma_u \cdot \sigma_v =\sum_{w \in \Sn} c_{uv}^w \, \sigma_w, \ \ \ c_{uv}^w \in \Z.
\end{align*}
Computations of the structure constants $c_{uv}^w$ is called \textit{Schubert calculus} on the flag variety $Fl_n $. Geometrically, $c_{uv}^w$ is the intersection number $\int_{Fl_n } (\sigma_u \cdot \sigma_v \cdot \sigma_{w_0w})$, and this implies the positivity for the structure constants, i.e., $c_{uv}^w \geq 0$ by Kleiman's transversality theorem (see e.g.\ \cite[Sect.~1.3]{Bri}).

Motivated by this, Harada and Tymoczko considered the following problem in \cite{HaTy}. Let $p_w \in H^*(\Pet{n};\C)$ denote the image of the Schubert class $\sigma_w \in H^*(Fl_n ;\C)$ under the restriction map $H^*(Fl_n ;\C) \to H^*(\Pet{n};\C)$. They called $p_w$ the \emph{Peterson Schubert class} corresponding to $w$. Since the restriction map $H^*(Fl_n ;\C) \to H^*(\Pet{n};\C)$ is surjective (\cite{HaTy,Insko}), it is natural to ask whether there exists a natural subset of Peterson Schubert classes $p_w$ which forms an additive basis of $H^*(\Pet{n};\C)$. They gave an answer to this question as follows.
Let $J = \{j_1 < j_2 < \dots < j_m \}$ be a subset of $[n-1]$. They defined the element $v_J \in \mathfrak{S}_n$ to be the product of simple transpositions whose indices are in $J$, in increasing order, that is,
\begin{align}\label{eq: def of vJ}
 v_J \coloneqq s_{j_1} s_{j_2} \cdots s_{j_m}. 
\end{align}

\begin{theorem} $($\cite[Theorem~4.12]{HaTy}$)$ \label{Theorem_HaradaTymoczko_basis}
The set $\{p_{v_J} \mid J \subseteq [n-1] \}$ forms a $\C$-basis of $H^*(\Pet{n};\C)$.
\end{theorem}

By this theorem, we may expand the product $p_{v_J} \cdot p_{v_K}$ in terms of the Peterson Schubert classes $p_{v_L}$:
\begin{align} \label{eq:Schubert_calculus_Pet} 
p_{v_J} \cdot p_{v_K} =\sum_{L \subseteq [n-1]} \c_{JK}^L \, p_{v_L}, \ \ \ \c_{JK}^L \in \C.
\end{align}
Computing the structure constants $\c_{JK}^L$ is called \textit{Peterson Schubert calculus} in \cite{GoGo}. Harada and Tymoczko also gave Monk's formula for $\c_{JK}^L$ in \cite[Theorem~6.12]{HaTy}, which is the case for $|J|=1$. Recently, Goldin and Gorbutt gave combinatorial formulas for the structure constants $c_{JK}^L$ in \cite[Theorems~1,4,6,7]{GoGo} which are manifestly positive and integral. In particular, their formulas imply the positivity for the structure constants.

\begin{theorem} $($\cite[Corollary~8]{GoGo}$)$ \label{theorem:GoGo}
The structure constants $\c_{JK}^L$ in \eqref{eq:Schubert_calculus_Pet} are non-negative integers for all $J,K,L\subseteq[n-1]$.
\end{theorem}

This theorem ensures that all the coefficients $\c_{JK}^L$ in \eqref{eq:Schubert_calculus_Pet}  are (non-negative) integers, but it is not obvious whether $\{p_{v_J}\in H^{2|J|}(\Pet{n};\Z) \mid J \subseteq [n-1]\}$ forms a $\Z$-basis of $H^*(\Pet{n};\Z)$. Moreover, it is natural to ask a geometric reason of this positivity for the structure constants $\c_{JK}^L$ (cf.\ \cite[Remark~3.4]{BaHa} and \cite[p.43, question~(2)]{HaTy}). In what follows, we give an answer to this question. Recall from \cite{BaHa} that we have Giambelli's formula for the Peterson Schubert classes.

\begin{theorem}$($Giambelli's formula for the Peterson variety, \cite[Theorem~3.2]{BaHa}$)$ \label{theorem_Giambelli}
For $J\subseteq[n-1]$, we have 
\begin{align} \label{eq:Giambelli}
 p_{v_J} = \frac{1}{|J_1| ! |J_2| ! \cdots |J_m| !} \prod_{i\in J} p_{s_i} ,
\end{align}
where $J_{k} \ (1 \leq k \leq m)$ are the the connected components of $J$.
\end{theorem}

\begin{remark}
Drellich gave Giambelli's formula for arbitrary Lie types in \cite{Dre2}.
\end{remark}

As is well-known, the Schubert class $\sigma_{s_i}$ can be written as $\sigma_{s_i} = x_1+\cdots+x_i=\varpi_i$ in $H^2(Fl_n;\Z)$, where $x_1,\ldots,x_n$ are defined in \eqref{eq:xi}. This implies that 
\begin{align*}
 p_{s_i}=\varpi_i \quad \text{in $H^2(\Pet{n};\Z)$},
\end{align*}
for $1\le i\le n-1$ by taking the restriction. Thus, the right hand side of \eqref{eq:Giambelli} is nothing but $\varpi_J$ in Definition~\ref{defi:varpi}.
As a consequence of Theorems~\ref{thm: Z-basis of cohomology of Peterson} and \ref{theorem_Giambelli}, we obtain the following result which explains the geometric background of the Peterson Schubert calculus.

\begin{corollary} \label{corollary:AHKZ}
For $J\subseteq[n-1]$, we have $p_{v_J} = \varpi_J$. 
In particular, the set \[\{p_{v_J} \in H^{2|J|}(\Pet{n};\Z) \mid J \subseteq [n-1] \}\]
forms a $\Z$-basis of $H^*(\Pet{n};\Z)$.
Moreover, the structure constant $\c_{JK}^L$ in \eqref{eq:Schubert_calculus_Pet} is equal to the structure constant $\d_{JK}^L$ in Theorem~$\ref{theorem:LeftRight_diagram}$.
\end{corollary}

\begin{remark}
This implies that Lemma~$\ref{lemm:AHKZ}$ is essentially a special case of Monk's formula \cite[Theorem~$6.12$]{HaTy}. 
\end{remark}

\vspace{10pt}

By Corollary~\ref{corollary:AHKZ}, the structure constants $\d_{JK}^L$ can also be computed by the formulas for $\c_{JK}^L$ proved earlier by Goldin-Gorubtt \cite{GoGo} in the $\C^{\times}$-equivariant setting (see Section~\ref{subsect: combi}).
Their approach to the structure constants is mostly combinatorial whereas our approach is geometric based on the properties of $X_J$ and $\Omega_J$. We end this subsection by giving a short observation on the difference of their formulas and ours.

Suppose that $J,K,L\subseteq[n-1]$ are all connected subsets such that $J\cup K\subseteq L$, $|L|=|J|+|K|$. Then, we may write $J=[a_1,a_2]$, $K=[b_1,b_2]$, $L=[c_1,c_2]$, and we may assume that $a_1\le b_1$ by interchanging the roles of $J$ and $K$ if necessary.
In this case, 
their formula (\cite[Corollary~2]{GoGo}) for $\c_{JK}^L$ is quite simple: 
\begin{align*}
 \c_{JK}^L = \binom{a_2-b_1+1}{a_1-c_1}\binom{b_2-a_1+1}{b_1-c_1}.
\end{align*}
For general $J,K,L\subseteq[n-1]$, their computation of $\c_{JK}^L$ consists of three (ordered) formulas (\cite[Theorems 3, 5, 6]{GoGo}) each of which successively makes a reduction to the computations in the former case.

In contrast, our formula has several terms even when $J,K,L\subseteq[n-1]$ are all connected, however it provides a single formula which covers all the cases of general $J,K,L\subseteq[n-1]$.

\vspace{10pt}

\subsection{Relations to mixed Eulerian numbers}
We next explain the relations of the results in this paper to the works on mixed Eulerian numbers introduced and studied by Postnikov (\cite{Postnikov09}).

We briefly recall the definition of mixed Eulerian numbers.
For $a_1,\ldots,a_n\in\R^n$, the permutohedron $P_n(a_1,\ldots,a_n)$ is defined to be the convex hull of the $\Sn$-orbits of $(a_1,\ldots,a_n)$ in $\R^n$:
\begin{align*}
 P_n(a_1,\ldots,a_n) = \text{ConvexHull}\{(a_{w(1)},\ldots,a_{w(n)})\in\R^n\mid w\in \Sn\}.
\end{align*}
This is at most $(n-1)$-dimensional, and it sits inside of an affine hyperplane in $\R^n$. The $(n-1)$-dimensional volume (computed by projecting down to $\R^{n-1}$) of $P_n(a_1,\ldots,a_n)$ in terms of $u_i=a_i-a_{i+1}$ for $1\le i\le n-1$ can be written as
\begin{align*}
 \Vol P_n(a_1,\ldots,a_n) = \sum_{c_1,\ldots,c_{n-1}} A_{c_1,\ldots,c_{n-1}} \frac{u_1^{c_1}}{c_1!}\cdots \frac{u_{n-1}^{c_{n-1}}}{c_{n-1}!} ,
\end{align*}
where the sum is taken over all non-negative integers $c_1,\ldots,c_{n-1}$ with $c_1+\cdots+c_{n-1}=n-1$.
The coefficients $A_{c_1,\ldots,c_{n-1}}$ are called \textit{mixed Eulerian numbers} which are known to be non-negative integers (see \cite{Postnikov09} for details).\\

In \cite{Berget-Spink-Tseng}, Berget-Spink-Tseng studied log-concavity of matroid $h$-vectors in relation to mixed Eulerian numbers.
For that purpose, they considered the invariant subring of the Chow ring of the permutohedral variety with respect to the action of the symmetric group.
They introduced a basis $\delta_S$ of this invariant subring, and they proved that the structure constants of this basis can be written as products of mixed Eulerian numbers (\cite[Proposition~7.7 and Corollary 7.9]{Berget-Spink-Tseng}).
This invariant subring is known to be isomorphic to $H^*(\Pet{n};\Z)$ by \cite[Theorem~1.1]{AZ} (cf.\ \cite[Theorem B]{AHHM} for $\Q$-coefficients), and one can see that their basis corresponds to $\varpi_J$ in $H^*(\Pet{n};\Z)$ (compare \cite[Corollary~7.9]{Berget-Spink-Tseng} and Lemma~$\ref{lemm:AHKZ}$ in this paper).
Therefore, our formula (Theorem~\ref{theorem:LeftRight_diagram}) can also be regarded as computing some products of mixed Eulerian numbers by using the geometry of $\Pet{n}$.

Nadeau-Tewari (\cite{Nadeau-Tewari}) also found a relation between mixed Eulerian numbers and intersection numbers of Schubert varieties and the permutohedral variety for an arbitrary Lie type. After \cite{Berget-Spink-Tseng} and \cite{Nadeau-Tewari}, the second author of this paper investigated in \cite{Horiguchi21} a connection between Peterson Schubert calculus and mixed Eulerian numbers. More precisely, it was shown that the mixed Eulerian numbers can be written as intersection numbers of Schubert divisors in Peterson variety for an arbitrary Lie type (\cite[Theorem~1.1]{Horiguchi21}). 
We remark that, for type A, this formula was proved in \cite{Berget-Spink-Tseng} and \cite{Nadeau-Tewari} independently. 
Including this paper, all of these works are done independently, and these established connections between Peterson Schubert calculus and mixed Eulerian numbers. 

To end this paper, let us lastly deduce the formula for $d_{JK}^L$ in terms of mixed Eulerian numbers in the context of Peterson Schubert calculus.
For $J,K,L\subseteq[n-1]$, recall from \eqref{eq: djkl in terms of pairing} that we have
\begin{align*}
\d_{JK}^L 
&= \langle [X_{L}] , \varpi_{J} \cdot \varpi_{K} \rangle_{\Pet{n}} 
= \frac{1}{m_J}
\frac{1}{m_K} \int_{X_L} \left(\prod_{j\in J}\varpi_j\right)\left(\prod_{k\in K}\varpi_k\right)
\end{align*}
if $|J|+|K|=|L|$ and that we have $d_{JK}^L=0$ if $|J|+|K|\ne |L|$.
Taking the decomposition $L=L_1\sqcup \cdots \sqcup L_q$ into the connected components of $L$, we have $X_L=\prod_{i=1}^q X_{L_i}$ by Corollary~\ref{lem: decomp of PetJ}.
Hence, the integration over $X_L$ above can be written as a product of integrations over $X_{L_i}$ for $1\le i\le q$ ; 
\begin{align*}
\int_{X_L} \left(\prod_{j\in J}\varpi_j\right)\left(\prod_{k\in K}\varpi_k\right)
=
\prod_{i=1}^q\ 
\int_{X_{L_i}} \!\!\left(\prod_{j\in J\cap L_i}\varpi_j\right)\left(\prod_{k\in K\cap L_i}\varpi_k\right).
\end{align*}
Denoting $\ell_i\coloneqq |L_i|+1$, 
we have $X_{L_i}\cong \Pet{\ell_i}$ by Corollary~\ref{lem: decomp of PetJ} again.
Namely, each integration in the last equality is an intersection number of divisors on $\Pet{\ell_i}$.
We note that under this isomorphism $\varpi_r\in H^*(\Pet{\ell_i};\Q)$ $(1\le r\le |L_i|)$ corresponds to $\varpi_{r+\min L_i -1}\in H^*(X_{L_i};\Q)$ since we have $\Pet{\ell_i}\subseteq Fl(\C^{\ell_i})$ and $X_{L_i}\subseteq \Pet{n}\subseteq Fl(\C^n)$.
As explained above, the second author gave a formula which computes those intersection numbers as mixed Eulerian numbers (\cite[Theorem~1.1]{Horiguchi21}).
By applying it to the integrations above, we obtain the following formula for which we take the convention that $A_{c_1,\ldots,c_{p}}=0$ unless $c_1+\cdots+c_{p}=p$ for positive integers $p$.\\

\begin{theorem}\label{thm: mixed Eulerian}
For $J,K,L\subseteq[n-1]$, we have 
\begin{align*}
 d_{JK}^L = 
 \frac{1}{m_J}\frac{1}{m_K} \prod_{i=1}^q A_{c^{(i)}_1,\ldots,c^{(i)}_{\ell_i -1}},
\end{align*}
where $L=L_1\sqcup \cdots \sqcup L_q$ is the decomposition into the connected components of $L$ and $c^{(i)}_1,\ldots,c^{(i)}_{\ell_i -1}$ are the multiplicities of the product $(\prod_{j\in J\cap L_i}\varpi_j)(\prod_{k\in K\cap L_i}\varpi_k)$ given by
\begin{align*}
 c^{(i)}_r \coloneqq
 \begin{cases}
  2 \quad &\text{if $r +\min L_i -1\in J\cap K$}, \\
  1 &\text{if $r +\min L_i -1\in (J\cup K)-(J\cap K)$}, \\
  0 &\text{otherwise}
 \end{cases}
\end{align*}
for $1\le i\le q$ and  $1\le r\le |L_i|$ $($which means $r +\min L_i -1\in L_i$$)$. 
\end{theorem}

\begin{remark}
As we noted above, this formula can also be deduced from \cite[Proposition~7.7]{Berget-Spink-Tseng}.
\end{remark}

\begin{remark}
The indexes of the mixed Eulerian numbers appearing in Theorem~$\ref{thm: mixed Eulerian}$ are always less than or equal to $2$. 
In \cite{Berget-Spink-Tseng} and \cite{Horiguchi21}, mixed Eulerian numbers with arbitrary indexes are considered.
\end{remark}

\end{document}